\documentclass[smallextended]{svjour3} 
\smartqed 
\usepackage{graphicx}

\usepackage{amsfonts}
\usepackage{graphicx}
\usepackage{overpic}
\usepackage{epstopdf}
\usepackage{algorithm}
\ifpdf
  \DeclareGraphicsExtensions{.eps,.pdf,.png,.jpg}
\else
  \DeclareGraphicsExtensions{.eps}
\fi

\usepackage{amssymb}
\usepackage{amsmath}
\usepackage{amstext}
\usepackage{enumerate}
\usepackage{xcolor}

\newtheorem{defn}{Definition}
\newtheorem{assumption}{Assumption}

\DeclareMathOperator*{\minimize}{minimize}
\DeclareMathOperator*{\argmin}{arg\;min}
\DeclareMathOperator*{\esssup}{ess\;sup}
\DeclareMathOperator*{\essinf}{ess\;inf}
\def\Es{\mathrm{Es}}
\def\rEs{\mathrm{Es}^\prime}
\def\rV{V^\prime}
\def\rOmega{\Omega^\prime}
\def\rtheta{\theta^\prime}
\def\RR{\mathbb{R}}
\def\PP{\mathbb{P}}

\def\X{\mathcal{X}}

\def\Bincdf{\Phi}
\def\Varcdf{\Psi}
\def\q0{r}
\def\qstar{r^\ast}
\def\istar{i^\ast}
\def\xsol{\hat{x}}
\def\qsol{\hat{q}}

\def\pprior{p_{\textit{prior}}}
\def\ppost{p_{\textit{post}}}
\def\ptrial{p_{\textit{trial}}}
\def\Ntrial{N_{\textit{trial}}}

\def\defeq{\mathrel{\mathop:}=}
\allowdisplaybreaks[4]

\def\vv{\vskip\baselineskip}

\numberwithin{equation}{section}

\begin{document}

\title{Chance-constrained optimization with tight confidence bounds} 
\author{Mark Cannon}

\institute{Engineering Science Dept., University of Oxford, OX1 3PJ, UK. mark.cannon@eng.ox.ac.uk}

\maketitle

\begin{abstract}
Convex sample approximations of chance-constrained optimization problems are considered, in which chance constraints are replaced by sets of sampled constraints. 
We propose a randomized sample selection strategy that allows tight bounds to be derived on the probability that the solution of the sample approximation is feasible for the original chance constraints.
These confidence bounds are shown to be tighter than the bounds that apply if samples are selected according to optimal or greedy discarding strategies. 
We further show that the same confidence bounds apply to solutions that are obtained from a two stage process in which a sample approximation of a chance-constrained problem is solved, then an empirical measure of the violation probability of the solution is obtained by counting the number of violations of an additional set of sampled constraints. 
We use this result to design a repetitive scenario approach that meets required tolerances on violation probability given any specified a priori and a posteriori probabilities. 
These bounds are tighter than confidence bounds available for previously proposed repetitive scenario approaches, and we show that the posterior bounds are exact for a particular problem subclass. 
The approach is illustrated through numerical examples, and extensions to problems involving multiple chance constraints are discussed.
\end{abstract}


\section{Introduction}

An important class of optimization problems involves chance constraints, namely constraints dependent on stochastic parameters, which are required to hold with specified probabilities.
Solution methods and applications for optimization under chance constraints were first considered in the context of problems in economics and management~\cite{charnes59,charnes63}. More recently, chance-constrained optimization has been applied to diverse problems in
finance~\cite{gaivoronski05,boyd16}, process design~\cite{wendt02,calafiore06},
 model predictive control~\cite{cannon09,schildbach14} and building control~\cite{ma12}. Further applications and references are discussed in~\cite{ahmed08}.

For problems with constraints that may be violated up to prescribed limits on violation probability, chance constraints are less stringent than their robust counterparts, which impose constraints for all realizations of uncertainty. However, methods of handling chance constraints using explicit probability distributions can lead to intractable optimization problems. This motivates the use of scenario or sample-based methods in which
constraints are imposed for finite sets of independent samples of the uncertain parameters. These approaches have the advantages that convexity is preserved, assuming that the constraints are convex in the decision variables for all uncertainty realizations, and that probabilistic bounds can be determined on the confidence with which the solution satisfies constraints~\cite{calafiore05,calafiore06,campi08}.
%
%
%
In order to keep computation
within practicable limits, it is important to understand how the sample size affects the accuracy with which the solution of the sampled problem approximates the solution of the chance-constrained problem.

The seminal papers~\cite{campi08,calafiore10} provide bounds on the confidence that a decision variable satisfies a chance constraint, conditioned on the fact that the decision variable is optimal for a sampled problem in which the chance constraint is replaced by a randomly extracted set of sampled constraints.
These bounds are tight in the sense that they cannot be improved without additional information about the sampled problem, and they are exact (i.e.~they coincide with the actual distribution of violation probabilities) for a particular problem subclass.
%
However, since this approach assumes that the sampled problem invokes the entire set of sampled constraints, a high level of confidence necessitates a large number of samples and hence a low probability of constraint violation. Consequently the approach cannot generally
deliver a high degree of accuracy of approximation for problems involving chance constraints with violation probabilities that are not close to zero.

Alternative formulations, which are more suitable for approximating chance constraints with arbitrary violation probabilities, use a certain proportion of parameter samples to define a set of sampled constraints and discard the rest.
Bounds on the confidence with which a solution of the resulting sampled problem satisfies a given chance constraint are derived in~\cite{calafiore10,campi11}.
However, these bounds are obtained under the assumption that sampled contraints are discarded optimally with respect to the objective function, or that constraint selection heuristics are used to approximate an optimal sample discarding procedure. The solutions of the sampled problem may therefore have poor generalization properties, and we show in this paper that, relative to a randomized sample discarding procedure, this leads to a lower confidence of satisfying the underlying chance constraints.  


A third type of scenario approach for chance-constrained programming 
selects a solution from among the solutions of a set of sampled problems, with constraints defined in each case by an independently extracted set of parameter samples~\cite{chamanbaz16,calafiore17}. By incorporating \textit{a posteriori} empirical constraint violation tests based on additional parameter samples that are not used in the definition of the sampled problem, this approach can potentially provide tighter bounds on the confidence of satisfying an associated chance constraint. However~\cite{chamanbaz16} does not provide \textit{a priori} confidence bounds, and the posterior confidence bounds (based on~\cite{alamo09}) are necessarily conservative for the convex setup employed here; on the other hand, the bounds given in~\cite{calafiore17} are non-conservative only for a particular problem subclass, as we show in this paper.

This paper explores and develops the connection between 
the confidence of chance constraint satisfaction for single-shot scenario approaches with and without sample discarding~\cite{campi08,calafiore10,campi11}, and repetitive scenario approaches~\cite{chamanbaz16,calafiore17}.
%
Considering the properties of problems with sampled constraints that are discarded \textit{at random} rather than according to a deterministic algorithm, we derive new bounds on the confidence that the solution is feasible for an associated chance-constrained problem.
These bounds are tight (in the sense that they cannot be improved without additional information on the problem), and, since a combinatorial factor appearing in the bounds of~\cite{calafiore10,campi11} is not required, they demonstrate that a considerable improvement in approximation accuracy can be achieved using a randomized sample discarding approach.
We discuss a procedure for implementing randomized sample discarding based on the repetitive scenario approach.
We describe how to determine the number of sampled problems that must be solved, and their respective sample sizes, in order to ensure specified tolerances on the violation probability of the solution for any given prior and posterior probabilities. The resulting posterior confidence bound coincides with that of~\cite{calafiore17} for a specific problem subclass, however, we show that it is in fact exact for this case and we give tight bounds in all other cases. 

The rest of the paper is organised as follows. Section~\ref{sec:problem_definition} gives the problem definition. Section~\ref{sec:results} gives the main results, then discusses related results on the optimal value of the objective function and comparisons with existing results for deterministic sample discarding procedures. Section~\ref{sec:algorithm} describes an algorithm for approximating the solution of a chance-constrained problem with specified prior and posterior confidence bounds on constraint violation probability. Two examples are given to illustrate the algorithm: an application to the problem of determining the smallest hypersphere containing a given probability mass, and an application to a finite horizon optimal control problem considered in~\cite{calafiore17}, generalized to the case of multiple chance constraints. Section~\ref{sec:conclusions} draws conclusions.




\section{Problem definition and assumptions}\label{sec:problem_definition}

Consider the chance-constrained optimization problem with decision
variable $x\in\RR^n$:
\begin{equation}\label{eq:ccp}
\begin{aligned}
  \minimize_{x\in\X}  & \ \ c^\top x \\
  \text{subject to} & \ \ \PP\bigl\{ f(x,\delta) > 0 \bigr\} \leq \epsilon .
\end{aligned}
\end{equation}
Here $\epsilon$ is a specified probability, $\delta \in \Delta
\subseteq \RR^d$ is a vector of random parameters and $\PP$ a probability measure defined on $\Delta$. 
Note that  $\epsilon$ can take any value in the interval $[0,1]$, and in particular $\epsilon$ is not assumed to be close to $0$.
The domain $\X\subset\RR^{n}$ of the decision variable $x$ and the function $f:\RR^{n} \times \Delta \to \RR$ satisfy the following assumption.

\begin{assumption}\label{assump:convexity}
For all $\delta\in\Delta$, $f(\cdot,\delta)$ is convex and lower-semicontinuous, and $\X$ is compact and convex.
\end{assumption}


The chance constraint $\PP
\{ f(x,\delta) > 0 \} \leq \epsilon$ is not necessarily convex in $x$
(except for certain special cases, see e.g.~\cite{prekopa95}), and problem (\ref{eq:ccp}) is therefore nonconvex in general. To avoid the computational difficulties associated with the chance constraint in (\ref{eq:ccp}), we consider an approximate problem formulation using samples of the uncertain parameter $\delta$.
%
Let $\omega_m = \{\delta^{(1)},\ldots,\delta^{(m)}\}\in\Delta^m$ denote a collection of $m$ independent identically distributed (i.i.d.) samples $\delta^{(i)}$ of the random variable~$\delta$. The sample indices are assumed to be statistically independent of the sample values, so that $\omega_{\q0} = \{\delta^{(1)},\ldots,\delta^{(\q0)}\}$ denotes a randomly selected subset of $\omega_m$, for any $\q0<m$.

To motivate a discussion of sample selection methods, we define (following~\cite{calafiore10,campi11}) the
sample approximation of~(\ref{eq:ccp}) with optimal sample discarding as
\begin{equation}\label{eq:sp}
\begin{aligned}
\minimize_{\substack{\omega \subseteq \omega_m \\ x \in\X}} & \ \ c^\top x \\
\text{subject to}  
& \ \  f(x,\delta) \leq 0 \ \text{ for all } \delta\in\omega \\
& \ \ \lvert \omega \rvert = q
\end{aligned}
\end{equation}
for a given integer $q$, with $n \leq q \leq m$. 
%
The optimal values of $\omega$ and $x$ in (\ref{eq:sp}) are denoted $\omega^\ast$ and $x^\ast(\omega^\ast)$ respectively.

Since $\X$ is compact and convex, and since $f(x,\delta)\leq 0$ is a convex constraint on $x$,
the sampled problem (\ref{eq:sp}) can be expressed as a mixed integer program with a convex continuous relaxation.
%
This implies that (\ref{eq:sp}) can be solved exactly using a branch
and bound approach (see e.g.~\cite{kall94}). However, unless $m$ is
large, the solution $x^\ast(\omega^\ast)$ can have poor generalization
properties
since, as discussed in Section~\ref{sec:results} of this paper, the available bounds
on the confidence that $x^\ast(\omega^\ast)$ satisfies the chance
constraint of (\ref{eq:ccp}) are not tight for general uncertainty distributions. Moreover the computation required to solve for $x^\ast(\omega^\ast)$ grows rapidly with $m$ and $m-q$. 

In this paper we therefore consider problems with sampled constraints that
are defined by randomly selecting subsets of the multisample $\omega_m$. Let 
\begin{equation}\label{eq:xdef}
x^\ast( \omega ) = \arg\min_{x\in\X} \ c^\top x 
\ \text{ subject to } \ 
f(x,\delta) \leq 0 
\ \text{ for all } \
\delta \in \omega
\end{equation}
%
and define $\Omega_q(\omega_m)$ for any given $q\leq m$ as the collection of all $q$-element subsets $\omega$ of $\omega_m$ such that $x^\ast(\omega)$ violates the constraint $f(x,\delta) \leq 0$ for all $\delta\in\omega_m\setminus\omega$, i.e.
\begin{equation}\label{eq:Omegadef}
\Omega_q(\omega_m) = \bigl\{ \omega \subseteq \omega_m :
\lvert \omega \rvert = q 
\text{ and } 
f\bigl( x^\ast(\omega), \delta \bigr) > 0
\text{ for all } 
\delta \in \omega_m\setminus\omega \bigr\} .
\end{equation}
We make the following assumptions on (\ref{eq:xdef}).

\begin{assumption}\label{ass:sp_feas}
  The optimization (\ref{eq:xdef}) is almost surely feasible for $\omega=\omega_m$.
\end{assumption}



%

\begin{assumption}\label{ass:greedy_feas}
The solution of (\ref{eq:xdef}) for any $\omega\subseteq\omega_m$ satisfies $f(x^\ast(\omega),\delta ) \neq 0$ for all $\delta\in \omega_m \setminus \omega$, with probability~$1$.
\end{assumption}

The feasibility requirement of Assumption~\ref{ass:sp_feas} is trivially satisfied if the robust optimization corresponding to~(\ref{eq:ccp}) with $\epsilon = 0$ is feasible. Clearly Assumption~\ref{ass:sp_feas} could be restrictive since it is equivalent to the requirement that, with probability~$1$, an $x$ exists satisfying the constraints of~(\ref{eq:xdef}) with $\omega=\omega_r$, for any $r\leq m$. We note, however, that the results of this paper could be extended to situations in which~(\ref{eq:xdef}) has a non-zero probability of infeasibility by using a framework for analysis such as that described in~\cite{calafiore10}.
%
%
On the other hand, by convexity Assumption~\ref{ass:greedy_feas} holds if and only if, for any $r<m$, problem (\ref{eq:xdef}) with $\omega=\omega_r$ is non-degenerate with probability~$1$ (i.e.~the dual of problem~(\ref{eq:xdef}) almost surely has a unique solution~\cite{fletcher00}).
%
%

In general $\Omega_q(\omega_m)$ may contain more than one subset of $\omega_m$. In the sequel we refer to each distinct $\omega\in\Omega_q(\omega_m)$ as a  \textit{level-$q$ subset of $\omega_m$} and to the corresponding solutions $x^\ast(\omega)$ of (\ref{eq:xdef}) for $\omega\in\Omega_q(\omega_m)$ as \textit{level-$q$ solutions}.

We define the essential (constraint) set, $\Es(\omega)$, of (\ref{eq:xdef}) for given $\omega$ as follows (see also \cite[Def.~2.9]{calafiore10}). 

\begin{defn}[Essential set]\label{def:ess_set}
$\Es(\omega) = \{\delta^{(i_1)},\ldots,\delta^{(i_k)}\}\subseteq\omega$ is an essential set of problem (\ref{eq:xdef}) if
\begin{enumerate}[(i).]
\item
$x\bigl(\{\delta^{(i_1)},\ldots,\delta^{(i_k)}\}\bigr)=x^\ast(\omega)$, and
\item
$x(\omega\setminus \delta) \neq x^\ast (\omega)$ for all $\delta \in \{\delta^{(i_1)},\ldots,\delta^{(i_k)}\}$. 
\end{enumerate}
\end{defn}


An essential set consists of samples $\delta\in\omega$ that are associated with active constraints at the solution $x^\ast(\omega)$ of (\ref{eq:xdef}).
If Assumptions~\ref{ass:sp_feas} and~\ref{ass:greedy_feas} hold, then $\Es(\omega)$ is necessarily uniquely determined by conditions (i) and (ii).
%
%
%
We define the \textit{maximum support dimension} of (\ref{eq:xdef}), denoted $\bar{\zeta}$, as the least upper bound that holds almost surely on the number of elements in the essential set of (\ref{eq:xdef}) for any size of multisample $\omega$, namely the maximum value of $\esssup_{\omega\in\Delta^m}\lvert \Es(\omega) \rvert$ over all finite integers $m\geq 1$.
%
It is easy to show that $\bar{\zeta}$ (which is equivalent to Helly's dimension for (\ref{eq:sp})~\cite[Def.~3.1]{calafiore10}) cannot be greater than $n$ if Assumptions~\ref{ass:sp_feas} and~\ref{ass:greedy_feas} hold.
Similarly we define the \textit{minimum support dimension} of (\ref{eq:xdef}), denoted $\underline{\zeta}$, as the minimum value of $\essinf_{\omega\in\Delta^m} \lvert \Es(\omega)\rvert$ for all finite $m\geq \bar{\zeta}$. Clearly $\underline{\zeta} \geq 0$ must hold for all problems.


\begin{assumption}\label{ass:support_dim}
The maximum and minimum support dimensions of (\ref{eq:xdef}) satisfy
$\esssup_{\omega\in\Delta^m}\lvert \Es(\omega) \rvert \leq \bar{\zeta}$ for all finite $m\geq 1$ 
and
$\essinf_{\omega\in\Delta^m}\lvert \Es(\omega) \rvert \geq \underline{\zeta}$ for all finite $m\geq \bar{\zeta}$
respectively, for some $\bar{\zeta}$ and $\underline{\zeta}$ such that $n \geq \bar{\zeta} \geq \underline{\zeta}\geq 0$.
\end{assumption}

\section{Main results}\label{sec:results}
The results presented in this section enable a randomized procedure to be constructed that ensures tight \textit{a priori} and \textit{a posteriori} bounds on the confidence of finding a solution of (\ref{eq:xdef}) that satisfies the chance constraint in~(\ref{eq:ccp}).  
For given $x\in\RR^n$, $V(x)$ denotes the violation probability
\[
V(x) = \PP\bigl\{ f(x,\delta) > 0\bigr\} .
\]
We first derive bounds on the conditional probability that $x^\ast(\omega)$ satisfies $V\bigl(x^\ast(\omega)\bigr) \leq \epsilon$ given that $\omega$ is a randomly selected level-$q$ subset of $\omega_m$. 
We then give bounds on the conditional probability that a level-$q$ solution $x^\ast(\omega)$ satisfies $V\bigl(x^\ast(\omega)\bigr) \leq \epsilon$ given that $x^\ast(\omega)$ is a randomly selected level-$\q0$ solution, for any given $\q0 \leq q$. These provide the basis for a posteriori bounds on  the confidence that $x^\ast(\omega)$ satisfies $V\bigl(x^\ast(\omega)\bigr) 
\leq \epsilon$.
Finally we provide bounds on the probability of generating a level-$q$ subset of $\omega_m$ using a randomized sample selection procedure; these bounds make it possible to determine a priori bounds on the probability of obtaining a level-$q$ solution.

The solution of a randomized optimization problem based on a finite multisample cannot in general satisfy with certainty the chance constraint in~(\ref{eq:ccp}). Instead we seek a solution $x\in\X$ such that the constraint violation probability $V(x)$ lies in a given interval, $(\underline{\epsilon},\bar{\epsilon}]$, with a specified level of confidence, $\pprior\in(0,1)$.
Two-sided confidence bounds are important in this context because the violation probability $V(x)$ should be close to $\epsilon$ with a high degree of confidence in order that $x$ approximates the solution of the chance-constrained problem (\ref{eq:ccp}) for any given $\epsilon\in[0,1]$. 

In order to define the probability of an event that depends on the multisample $\omega_m\in\Delta^m$, we use $\PP^m$ to denote the product measure on $\Delta^m$. The binomial distribution function is denoted
\[
\Bincdf(n;N,p) = \sum_{i=0}^n \binom{N}{i} p^i (1-p)^{N-i} ,
\] 
so that $\Bincdf(n;N,p)$ is the probability of $n$ or fewer events occuring in $N$ independent trials, each of which has probability $p$.


\begin{theorem}[Confidence bounds for level-$q$ solutions]
\label{thm:feas_bound}
For any $\epsilon\in [0,1]$ and integers $q$ and $m$ such that $\bar{\zeta} \leq q \leq m$ we have
\begin{equation}\label{eq:lower_bound_ineq}
{\PP}^m \bigl\{ V \bigl(x^\ast(\omega_q)\bigr) \leq \epsilon 
\ | \ \omega_q\in\Omega_q(\omega_m) \bigr\} 
\geq 
\Bincdf( q-\bar{\zeta}; m, 1-\epsilon ) ,
\end{equation}
and, for $\underline{\zeta}\leq q\leq m$,
\begin{equation}\label{eq:upper_bound_ineq}
{\PP}^m \bigl\{ V \bigl(x^\ast(\omega_q)\bigr\} \leq \epsilon 
\ | \ \omega_q\in\Omega_q(\omega_m) \bigr\}
\leq 
\Bincdf( q-\underline{\zeta}; m, 1-\epsilon ) .
\end{equation}
\end{theorem}

Theorem~\ref{thm:feas_bound} (which is proved in Section~\ref{sec:levelq-bounds}) 
provides upper and lower bounds on the 
probability of the event $V(x^\ast(\omega_q)) \leq \epsilon$,
conditioned on the event that $\Omega_q(\omega_m)$ contains $\omega_q$.
Since $\omega_q = \{\delta^{(1)},\ldots,\delta^{(q)}\}$ is statistically identical to a randomly selected subset of $\omega_m$ of cardinality $q$,
a direct consequence of Theorem~\ref{thm:feas_bound}
is that the probability of the event $V(x(\omega)) \leq \epsilon$,
given that $\omega$ is a randomly selected member of $\Omega_q(\omega_m)$, also satisfies the upper and lower bounds in~(\ref{eq:lower_bound_ineq}) and (\ref{eq:upper_bound_ineq}).

Whenever $q<m$ and $\bar{\zeta} > 1$, the lower confidence bound in~(\ref{eq:lower_bound_ineq}) is greater than the bound derived in~\cite{calafiore10,campi11} on the probability that the solution of~(\ref{eq:sp}) (in which samples are discarded optimally) satisfies $V(x^\ast(\omega^\ast))\leq \epsilon$; this is discussed in detail in Section~\ref{sec:comparison}. The resulting improvement in the confidence bound for a randomly selected level-$q$ solution is significant because a combinatorial factor that appears in the bounds of~\cite{calafiore10,campi11} is not required in (\ref{eq:lower_bound_ineq}). Furthermore, the confidence bounds of Theorem~\ref{thm:feas_bound} hold with equality if the support dimension of~(\ref{eq:xdef}) is unique, i.e.~if $\underline{\zeta} = \bar{\zeta}$. An example of this is the class of fully supported problems, for which $\underline{\zeta} = \bar{\zeta} = n$~\cite{calafiore10,campi11}.

For general values of $m$, $q$ and $n$, it is not computationally tractable to identify all level-$q$ subsets of the multisample $\omega_m$ and then select one at random in order to take advantage of the confidence bounds in Theorem~\ref{thm:feas_bound}. Clearly the optimal solution of~(\ref{eq:sp}), if available, could be used to identify
a level-$q$ subset 
(namely $\omega^\ast$), and likewise greedy constraint selection algorithms are able to identify suboptimal level-$q$ subsets (see e.g.~\cite[Sec.~5.1]{calafiore10}). However the deterministic constraint discarding strategies employed by these methods 
cannot be used to select an element of $\Omega_q(\omega_m)$ at random.

Instead we consider a randomized constraint selection strategy. This is based on the observation that the essential set of~(\ref{eq:xdef}) for a randomly chosen subset of $\omega_m$, such as $\omega_{\q0} = \{\delta^{(1)},\ldots,\delta^{(\q0)}\}$ for given $\q0$, is almost surely the essential set of a randomly selected level-$q$ subset, where $q$ is the number of elements of the multisample $\omega_m$ that satisfy the constraint $f(x^\ast(\omega_{\q0}),\delta) \leq 0$. Thus we can determine $q$ for given $\q0$ by first solving (\ref{eq:xdef}) for $\omega=\omega_{\q0}$, then counting the number of samples that satisfy $f(x^\ast(\omega_{\q0}),\delta) \leq 0$ and setting $q=\theta_{\q0}(\omega_m)$, where
\begin{equation}\label{eq:theta-def}
\theta_{\q0}(\omega_m) = \bigl\lvert \bigl\{ \delta\in\omega_m : f\bigl(x^\ast(\omega_{\q0}), \delta \bigr) \leq 0 \bigr\} \bigr\rvert .
\end{equation}
%
Since $\omega_{\q0}$ is statistically independent of the samples contained in $\omega_m\setminus \omega_{\q0}$, it can be shown that confidence bounds analogous to those provided by Theorem~\ref{thm:feas_bound} apply to $V(x^\ast(\omega_{\q0}))$. These bounds, which are stated in Theorem~\ref{thm:post_feas_bound} (and proved in Section~\ref{sec:priorpost-bounds}), provide a posteriori bounds on the confidence of constructing a solution of~(\ref{eq:sp}) with a specified constraint violation probability. 


\begin{theorem}[A posteriori confidence bounds conditioned on $\theta_{\q0}(\omega_m)$]
\label{thm:post_feas_bound}
For any $\epsilon\in [0,1]$  and integers $\q0$, $q$ and $m$ such that $\bar{\zeta}\leq \q0\leq q\leq m$ we have
\begin{equation}\label{eq:post_lower_bound_ineq}
{\PP}^m \bigl\{ V \bigl(x^\ast(\omega_{\q0})\bigr) \leq \epsilon 
\ | \ \theta_{\q0}(\omega_m) = q \bigr\} 
\geq 
\Bincdf( q-\bar{\zeta}; m, 1-\epsilon ) ,
\end{equation}
and, for $\underline{\zeta}\leq \q0\leq q\leq m$,
\begin{equation}\label{eq:post_upper_bound_ineq}
{\PP}^m \bigl\{ V \bigl(x^\ast(\omega_{\q0})\bigr) \leq \epsilon 
\ | \ \theta_{\q0}( \omega_m) = q \bigr\} 
\leq 
\Bincdf( q-\underline{\zeta}; m, 1-\epsilon ) .
\end{equation}
\end{theorem}

Theorem~\ref{thm:post_feas_bound} provides tight bounds on the conditional distribution of $V(x^\ast(\omega_{\q0}))$ given the value of $\theta_{\q0}(\omega_m)$. But $\theta_{\q0}(\omega_m)$ depends on the random sample $\omega_m$, and, for arbitrary $\q0$, there may be a only small probability that $\theta_{\q0}(\omega_m)$ lies in the required range in order that $x^\ast(\omega_{\q0})$ has, with a sufficiently high level of confidence, a constraint violation probability in the desired range. However, as we discuss in Section~\ref{sec:algorithm}, it is possible to choose the value of $\q0$ so as to maximize the probability that $\theta_{\q0}(\omega_m)$ lies in the required range. For this we make use of the following result (proof of which is given in Section~\ref{sec:priorpost-bounds}).

\begin{theorem}[Probability of selecting a level-$q$ subset of $\omega_m$]
\label{thm:prior_bound}
For any  integers $\q0$, $q$ and $m$ such that $\bar{\zeta} \leq \q0 \leq q \leq m$ we have
\begin{equation}\label{eq:prior_bound}
{\PP}^m \bigl\{ \theta_{\q0}(\omega_m) = q \bigr\}
\geq \binom{m-\q0}{q-\q0} \min_{\zeta\in[\underline{\zeta},\bar{\zeta}]} \frac{B(m-q+\zeta, q-\zeta+1)}{B(\zeta, \q0-\zeta+1)} .
\end{equation}
and
\begin{equation}\label{eq:prior_bound_fs}
{\PP}^m \bigl\{ \theta_{\q0}(\omega_m) = q \bigr\}
\leq \binom{m-\q0}{q-\q0} \max_{\zeta\in[\underline{\zeta},\bar{\zeta}]} \frac{B(m-q+\zeta, q-\zeta+1)}{B(\zeta, \q0-\zeta+1)} ,
\end{equation}
where $B(a,b)$ denotes the beta function for integers $a$ and $b$.
\end{theorem}

From Theorem~\ref{thm:prior_bound} it is possible to compute the value of $\q0$ that maximizes the probability with which $\theta_{\q0}(\omega_m)$ lies in any specified range. In conjunction with Theorem~\ref{thm:post_feas_bound}, this allows a priori bounds to be determined on the confidence that the violation probability $V(x^\ast(\omega_{\q0}))$ lies in the desired interval, $(\underline{\epsilon},\bar{\epsilon}]$. 
These confidence bounds make it possible to compute an upper bound on the number of times the procedure of solving (\ref{eq:xdef}) for $x^\ast(\omega_{\q0})$ and determining $\theta_{\q0}(\omega_m)$ must be repeated in order to ensure that a solution is obtained that satisfies $V(x^\ast(\omega_{\q0})) \in (\underline{\epsilon},\bar{\epsilon}]$ with a probability exceeding any given
a priori confidence level $\pprior$.
The proposed solution procedure, which is described in Section~\ref{sec:algorithm}, therefore meets a priori bounds on the confidence of determining a solution with a violation probability in the required range.

The computational cost of determining $\theta_{\q0}(\omega_m)$ in (\ref{eq:theta-def}) for given $x^\ast(\omega_{\q0})$ is typically small compared to the cost of solving (\ref{eq:xdef}) for $x^\ast(\omega_{\q0})$. 
For problems in which large numbers of samples are easy to obtain, this makes the use of large values of $m$ computationally attractive. In particular,
by using a large value of $m$ it is possible to obtain very sharp a posteriori confidence bounds. In such cases the main factor limiting a priori confidence bounds is the structural uncertainty in solutions of~(\ref{eq:sp}), namely the difference between the maximum and minimum support dimensions of~(\ref{eq:xdef}).


\subsection{Confidence bounds for randomly selected level-$q$ solutions}\label{sec:levelq-bounds}

The proof of Theorem~\ref{thm:feas_bound} is derived from the properties of a randomly selected level-$q$ subset of the multisample $\omega_m$. In order to simplify the analysis of problems with non-unique support dimensions (i.e.~with
$\underline{\zeta} \neq \bar{\zeta}$), we introduce a regularized version of the essential set that has a fixed cardinality.
To define this set we assign (similarly to~\cite{calafiore10}) a random label $\lambda^{(i)}$ to each sample $\delta^{(i)}$, where $\lambda^{(i)}$ is uniformly distributed on $[0,1]$ and independent of $\lambda^{(j)}$ for all $i\neq j$.
Furthermore, for any multisample $\omega_m = \{\delta^{(1)},\ldots,\delta^{(m)}\}$ with associated labels $\{\lambda^{(1)},\ldots,\lambda^{(m)}\}$ we define $\Lambda_k(\omega_m)$ as the subset of $\omega_m$ containing the $k$ samples with smallest  labels, so that $\lvert \Lambda_k (\omega_m) \rvert = k$ and $\delta^{(i)} \in \Lambda_k(\omega_m)$ if and only if $\lambda^{(i)} < \lambda^{(j)}$ for all $j$ such that $\delta^{(j)} \in \omega_m \setminus \Lambda_k(\omega_m)$.
The regularized essential set $\rEs_k(\omega_m)$ is defined for a given integer $k$ as follows.

\begin{defn}[Regularized essential set]\label{def:ress}
For a multisample $\omega$ and integer $k$ such that $0\leq k\leq \lvert \omega\rvert$, the regularized essential set is given by
\[
\rEs_k(\omega) = \begin{cases}
\Lambda_k \bigl(\Es(\omega)\bigr) 
& 
\text{if } k \leq \lvert \Es(\omega) \rvert \\
\Es(\omega) \cup \Lambda_\nu \bigl(\omega \setminus \Es(\omega)\bigr) 
&
\text{if } k > \lvert \Es(\omega) \rvert
\end{cases}
\]
where $\nu = k - \lvert \Es(\omega) \rvert$ and $\Es(\omega)$ is the essential set of~(\ref{eq:xdef}).
%
%
\end{defn}

Definition~\ref{def:ress} implies that $\lvert\rEs_k(\omega)\rvert = k$ almost surely, and that
\[
\rEs_{\underline{\zeta}}(\omega) \subseteq \Es(\omega) \subseteq \rEs_{\bar{\zeta}}(\omega).
\]
Using the regularized essential set we define a regularized version of violation probability:
\[
\rV_k(\omega) = \PP \bigl\{ \delta \in \rEs_k\bigl( \omega \cup \{\delta\}\bigr) \bigr\}. 
\]
Thus $\rV_k(\omega)$ is equal to the probability that, for given $k$, the regularized essential set associated with problem (\ref{eq:xdef}) changes when the multisample $\omega$ is extended to include a newly extracted sample $\delta\in\Delta$.
The regularized essential set also allows a regularized version of the set of level-$q$ subsets of $\omega_m$ to be defined for $k \leq q \leq m$ by
\begin{equation}
\label{eq:reg_lq_subsets}
\rOmega_{q,k} (\omega_m)
= \bigl\{ \omega \subseteq \omega_m : \lvert \omega \rvert = q 
\text{ and } 
\delta \in \rEs_k\bigl( \omega \cup \{ \delta\}\bigr)
\text{ for all } 
\delta \in \omega_m\setminus\omega \bigr\}  .
\end{equation}
Definition~\ref{def:ress} implies that $f(x^\ast(\omega),\delta) > 0$ whenever $\delta\in\smash{\rEs_{\underline{\zeta}} ( \omega \cup \{  \delta \} )}$,
and that $\delta\in\rEs_{\bar{\zeta}} ( \omega \cup \{ \delta \} )$ whenever $f(x^\ast(\omega),\delta) > 0$. These properties imply that
$\rV_{\underline{\zeta}}(\omega) \leq V\bigl(x^\ast(\omega)\bigr) \leq \rV_{\bar{\zeta}}(\omega)$,
and similarly
$\rOmega_{q,\underline{\zeta}} (\omega_m) \subseteq \Omega_q(\omega_m) \subseteq \rOmega_{q,\bar{\zeta}} (\omega_m)$.

This section finds (in Lemma~\ref{lem:selec_prob}) the probability that $\omega_{k}=\{\delta^{(1)},\ldots,\delta^{(k)}\}$ is equal to the regularized essential set $\rEs_k(\omega)$ of (\ref{eq:xdef}) for some $\omega\in\smash{\rOmega_{q,k}}(\omega_m)$, for any given $k$ and $q$ with 
$0\leq k\leq q\leq m$. 
This enables the conditional probability that $\rV_k(\omega_q) \leq \epsilon$ given that $\omega_q$ is an element of $\rOmega_{q,k}(\omega_m)$ to be determined (in Lemma~\ref{lem:feas_bound}), and the bounds in Theorem~\ref{thm:feas_bound} on the conditional probability of $V(x^\ast(\omega)) \leq \epsilon$ given that $\omega\in\Omega_q(\omega_m)$ are subsequently derived using this result.  
The approach is based on a fundamental result, stated in Lemma~\ref{lem:prior_prob}, on the distribution of the regularized violation probability $\rV_k(\omega_k)$. Related results are available in the literature (see e.g.~\cite[Eq.~3.2]{campi08} and~\cite[Eq.~3.11]{calafiore10}). 
The proofs of Lemmas~\ref{lem:prior_prob}, \ref{lem:selec_prob} and~\ref{lem:feas_bound} are provided in Appendix~\ref{apdx:proof_sec3}.

\begin{lemma}\label{lem:prior_prob}
For any $v\in[0,1]$ and integers $k$ and $m$ such that 
$0\leq k\leq m$ we have
\[
{\PP}^m\bigl\{ \rV_k (\omega_k) \leq v \bigr\} = v^{k} .
\]
\end{lemma}

 The probability that $\omega_k$ is equal to the regularized essential set, $\rEs_k(\omega)$ for some $\omega\in\rOmega_{q,k}(\omega_m)$, where $k\leq q\leq m$, is given by the following result.

\begin{lemma}\label{lem:selec_prob}
For any integers $k$, $q$ and $m$ such that 
$0\leq k\leq q \leq m$ we have
\[
\PP^m \bigl\{ \omega_k = \rEs_k(\omega) \cap \omega\in\rOmega_{q,k}(\omega_m) \bigr\} 
= k \binom{m-k}{q-k} B(m-q+k, q-k + 1) .
\]
\end{lemma}

The confidence bounds of Theorem~\ref{thm:feas_bound} can be established using the following lemma, which provides a subsidiary result on the regularized version of violation probability. The proof of this lemma is based on Lemmas~\ref{lem:prior_prob} and~\ref{lem:selec_prob}, and on the properties of a subset of $\omega_m$ selected at random from $\rOmega_{q,k}(\omega_m)$.

\begin{lemma}[Confidence bounds for regularized violation probabilities]
\label{lem:feas_bound}
For any $\epsilon\in [0,1]$ and integers $k$, $q$ and $m$ such that 
$0\leq k \leq q \leq m$ we have
\begin{equation}\label{eq:reg_lower_bound_ineq}
{\PP}^m \bigl\{ \rV_k (\omega_q) \leq \epsilon 
\ | \ \omega_q\in\rOmega_{q,k}(\omega_m) \bigr\} 
=
\Bincdf( q-k ; m, 1-\epsilon ) .
\end{equation}
\end{lemma}

The conditional distribution derived in Lemma~\ref{lem:feas_bound} for the regularized violation probability $\rV_k(\omega_q)$ given that $\omega_q\in \rOmega_{q,k}(\omega_m)$ is similar in form to the confidence bounds of Theorem~\ref{thm:feas_bound}. However the condition $\omega_q\in \Omega_{q}(\omega_m)$ employed in Theorem~\ref{thm:feas_bound} is in general much easier to check than membership of $\rOmega_{q,k}(\omega_m)$ because it requires only a function evaluation to check constraint violation rather than recomputation of the essential set of~(\ref{eq:xdef}), which is needed to determine whether $\delta\in\rEs_k(\omega_q \cup \{ \delta \} )$
when $k < \lvert \Es (\omega_q)\rvert$. 
In order to link Lemma~\ref{lem:feas_bound} to Theorem~\ref{thm:feas_bound} we make use of the independence of the regularized violation probability $\rV_k(\omega_q)$ and the cardinality of the essential set, $\lvert\Es(\omega_q) \rvert$. 
A consequence of this independence property is that the arguments used to prove Lemmas~\ref{lem:prior_prob}, \ref{lem:selec_prob} and \ref{lem:feas_bound} can be used to demonstrate that the same set of results hold when all probabilities are conditioned on the event that $\lvert \Es(\omega_q)\rvert$ takes any given value between $\underline{\zeta}$ and $\bar{\zeta}$.
For convenience we summarize the independence properties that are relevant to the proof of Theorem~\ref{thm:feas_bound} as follows.

\begin{lemma}\label{lem:ress_properties}
For any integers $k$ and $q$ such that $\underline{\zeta}\leq k\leq \bar{\zeta}$ and $\bar{\zeta}\leq q \leq m$ we have
\begin{equation}
\label{eq:reg_viol_independence}
{\PP}^m \bigl\{ \rV_k (\omega_{q}) \leq \epsilon 
\ \big| \ 
\lvert \Es(\omega_q)\rvert = k \bigr\}
=
{\PP}^m \bigl\{ \rV_k (\omega_{q}) \leq \epsilon \bigr\} 
\end{equation}
and
\begin{equation}
\label{eq:reg_lq_set_independence}
{\PP}^m \bigl\{ \omega_{q} \in \rOmega_{q,k}(\omega_m)  
\ \big| \ 
\lvert \Es(\omega_q)\rvert = k \bigr\}
=
{\PP}^m \bigl\{ \omega_{q} \in \rOmega_{q,k}(\omega_m)  \bigr\} .
\end{equation}
\end{lemma}

\begin{proof}
The sample labels $\{\lambda^{(1)},\ldots,\lambda^{(m)}\}$ are, by assumption, independent of $\lvert\Es(\omega_q)\rvert$. Therefore the probability of the event $\delta\in\rEs_k(\omega_q\cup\{\delta\})$ for a randomly extracted sample $\delta$ only depends on the values of $k$ and $q$, and does not depend on $\lvert\Es(\omega_q)\rvert$. Hence the events that $\rV_k(\omega_q)\leq \epsilon$ and $\omega_q\in\rOmega_{q,k}(\omega_m)$ are necessarily independent of the event that $\lvert\Es(\omega_q)\vert$ takes any given value. 
%
%
\qed\end{proof}



\noindent\textbf{Proof of Theorem~\ref{thm:feas_bound}}
From the observation that the events $\lvert\Es(\omega_q) \rvert = k$ are mutually exclusive and exhaustive for $k\in\{\underline{\zeta},\ldots,\bar{\zeta}\}$, we have
\begin{multline*}
{\PP}^m\bigl\{ V\bigl(x^\ast(\omega_q)\bigr) \leq \epsilon \ \big| \ \omega_q\in\Omega_q(\omega_m) \bigr\} 
\\
=
\sum_{k=\underline{\zeta}}^{\bar{\zeta}} {\PP}^m \bigl\{  V\bigl(x^\ast(\omega_q)\bigr) \leq \epsilon \ \big| \ \omega_q\in\Omega_q(\omega_m)  \cap \lvert \Es(\omega_q) \rvert = k\bigr\} \\
.\, {\PP}^m\bigl\{ \lvert \Es(\omega_q) \rvert = k \ \big| \ \omega_q\in\Omega_q(\omega_m)\bigr\} .
\end{multline*}
However the definitions of regularized and non-regularized violation probabilities and essential sets imply that $V(x^\ast(\omega_q))=\rV_k(\omega_q)$ and $\Omega_q(\omega_m) = \rOmega_{q,k}(\omega_m)$ if $\lvert \Es(\omega_q)\rvert = k$, and it follows that the event $V(x^\ast(\omega_q)) \leq \epsilon$ conditioned on $\omega_q\in\Omega_q(\omega_m)$ and $\lvert\Es(\omega_q)\rvert = k$ is identical to the event $\rV_k(\omega_q)\leq \epsilon$ conditioned on $\omega_q\in\rOmega_{q,k}(\omega_m)$ and $\lvert\Es(\omega_q)\rvert = k$. Therefore
\begin{multline*}
{\PP}^m\bigl\{ V\bigl(x^\ast(\omega_q)\bigr) \leq \epsilon \ \big| \ \omega_q\in\Omega_q(\omega_m) \bigr\} 
\\
=
\sum_{k=\underline{\zeta}}^{\bar{\zeta}} {\PP}^m \bigl\{  \rV_k(\omega_q) \leq \epsilon \ \big| \ \omega_q\in\rOmega_{q,k}(\omega_m)  \cap \lvert \Es(\omega_q) \rvert = k\bigr\} 
\\
.\, {\PP}^m\bigl\{ \lvert \Es(\omega_q) \rvert = k \ \big| \ \omega_q\in\Omega_q(\omega_m)\bigr\} ,
\end{multline*}
and the bounds in Theorem~\ref{thm:feas_bound} are derived from this expression using Lemmas~\ref{lem:feas_bound} and~\ref{lem:ress_properties}. Specifically, from (\ref{eq:reg_viol_independence}) and (\ref{eq:reg_lq_set_independence}) 
it follows that
\[
{\PP}^m \bigl\{  \rV_k(\omega_q) \leq \epsilon \ \big| \ \omega_q\in\rOmega_{q,k}(\omega_m)  \cap \lvert \Es(\omega_q) \rvert = k\bigr\} 
=
{\PP}^m\bigl\{ \rV_k(\omega_q) \leq \epsilon \ \big| \ \omega_q\in\rOmega_{q,k}(\omega_m) \bigr\} ,
\]
and using~(\ref{eq:reg_lower_bound_ineq}) we obtain
\begin{align*}
&{\PP}^m\bigl\{ V\bigl(x^\ast(\omega_q)\bigr) \leq \epsilon \ \big| \ \omega_q\in\Omega_q(\omega_m) \bigr\} 
\\
&\quad =
\sum_{k=\underline{\zeta}}^{\bar{\zeta}} {\PP}^m \bigl\{  \rV_k(\omega_q) \leq \epsilon \ \big| \ \omega_q\in\rOmega_{q,k}(\omega_m)  \bigr\} 
 {\PP}^m\bigl\{ \lvert \Es(\omega_q) \rvert = k \ \big| \ \omega_q\in\Omega_q(\omega_m)\bigr\} 
\\
&\quad =
 \sum_{k=\underline{\zeta}}^{\bar{\zeta}} \Phi(q-k; m, 1-\epsilon)
  {\PP}^m\bigl\{ \lvert \Es(\omega_q) \rvert = k \ \big| \ \omega_q\in\Omega_q(\omega_m)\bigr\} ,
\end{align*}
and hence the bounds $\Phi(q-\bar{\zeta}; m , 1-\epsilon) \leq \Phi(q-k; m , 1-\epsilon) \leq \Phi(q-\underline{\zeta}; m , 1-\epsilon)$ (which hold for all $\underline{\zeta}\leq k \leq \bar{\zeta}$, $q\leq m$ and $\epsilon\in[0,1]$) imply~(\ref{eq:lower_bound_ineq}) and~(\ref{eq:upper_bound_ineq}) since we must have 
$
\sum_{k=\underline{\zeta}}^{\bar{\zeta}}  {\PP}^m\bigl\{ \lvert \Es(\omega_q) \rvert = k \ \big| \ \omega_q\in\Omega_q(\omega_m)\bigr\} = 1
$.
\qed\vv

\subsection{The probability of selecting a level-$q$ subset and \textit{a posteriori} confidence bounds}\label{sec:priorpost-bounds}

This section provides proofs for Theorems~\ref{thm:post_feas_bound} and~\ref{thm:prior_bound}. 
%
%
These results are used in Section~\ref{sec:algorithm} to determine a randomized constraint selection strategy that is the basis of a method of determining solutions of (\ref{eq:ccp}) with specified a priori and a posteriori confidence bounds.
%
The posterior confidence bounds in Theorem~\ref{thm:post_feas_bound} are derived by an extension of the argument used in the proof of Theorem~\ref{thm:feas_bound}.
We then consider Theorem~\ref{thm:prior_bound}, which provides bounds on the probability that the number of samples $\delta\in\omega_m$ that satisfy $f(x^\ast(\omega_{\q0}),\delta) \leq 0$ is equal to a given value $q$.
The proof of this is based on a subsidiary result given in Lemma~\ref{lem:prior_bound}.
%

\vv
\noindent\textbf{Proof of Theorem~\ref{thm:post_feas_bound}}
The definition of $\theta_{\q0}$ in~(\ref{eq:theta-def}) implies that $\theta_{\q0}(\omega_m) = q$ if and only if, for some $\omega\in \omega_m\setminus \omega_{\q0}$, the set $\omega_{\q0} \cup \omega$ belongs to $\Omega_q(\omega_m)$. But the samples contained in $\omega_m\setminus \omega_{\q0}$ are statistically identical and independent of the samples in $\omega_{\q0}$, and hence any event conditioned on $\theta_{\q0}(\omega_m) = q$ is identical to the same event conditioned on $\omega_{\q0}\cup \{\delta^{(\q0+1)},\ldots,\delta^{(q)}\}=\omega_q\in\Omega_q(\omega_m)$. In particular we have
\[
{\PP}^m\bigl\{ V \bigl( x^\ast(\omega_{\q0})\bigr) \leq \epsilon \ | \ \theta_{\q0}(\omega_m) = q \bigr\} 
= {\PP}^m \bigl\{ V \bigl(x^\ast(\omega_{\q0})\bigr) \leq \epsilon \ | \ \omega_q\in\Omega_q(\omega_m) \bigr\} .
\]
Furthermore, whenever $\theta_{\q0}(\omega_m) = q$ and $\omega_q\in\Omega_q(\omega_m)$, we must have $V(x^\ast(\omega_{\q0})) = V(x^\ast(\omega_q))$, and hence
\[
{\PP}^m\bigl\{ V \bigl( x^\ast(\omega_{\q0})\bigr) \leq \epsilon \ | \ \theta_{\q0}(\omega_m) = q \bigr\} 
= {\PP}^m \bigl\{ V \bigl(x^\ast(\omega_{q})\bigr) \leq \epsilon \ | \ \omega_q\in\Omega_q(\omega_m) \bigr\} .
\]
The bounds~(\ref{eq:post_lower_bound_ineq}) and~(\ref{eq:post_upper_bound_ineq}) then follow from Theorem~\ref{thm:feas_bound}.
\qed\vv

To demonstrate the bounds of Theorem~\ref{thm:prior_bound}
we define $\rtheta_{\q0,k}(\omega_m)$ (analogously to $\theta_{\q0}(\omega_m)$ in~(\ref{eq:theta-def}))
as the number of samples, $\delta\in \omega_m$ with the property that the regularized essential set $\rEs_k(\omega_{\q0} \cup \{\delta\})$ is identical to $\rEs_k(\omega_{\q0})$, or equivalently
\begin{equation}\label{eq:reg_lq_set_def}
\rtheta_{\q0,k}(\omega_m) = r + \bigl\lvert \bigl\{ \delta\in\omega_m\setminus\omega_{\q0} : 
\delta \notin \rEs_k\bigl(\omega_{\q0} \cup \{ \delta \}\bigr) \bigr\} \bigr\rvert .
\end{equation}
%
With this definition we have the following result (see Appendix~\ref{apdx:proof_sec3} for a proof).

\begin{lemma}[Probability of selecting a regularized level-$q$ subset of $\omega_m$]
\label{lem:prior_bound}
For any integers $k$, $\q0$, $q$, $m$ satisfying $0\leq k\leq \q0 \leq q \leq m$ we have
\begin{equation}\label{eq:prior_bound_proof_eq2}
\PP^m \bigl\{ \rtheta_{\q0,k}(\omega_m) = q \bigr\} 
= 
\binom{m-r}{q-r} \frac{B(m-q+k,q - k+1)}{B(k, \q0-k+1)} . 
\end{equation}
\end{lemma}

\noindent\textbf{Proof of Theorem~\ref{thm:prior_bound}}
To derive the bounds (\ref{eq:prior_bound})-(\ref{eq:prior_bound_fs}), we note that
if $\lvert \Es(\omega_{\q0})\rvert = k$, then
$\theta_{\q0}(\omega_m) = \rtheta_{\q0,k}(\omega_m)$, and hence
\[
{\PP}^m \bigl\{ \theta_{\q0}(\omega_m) = q \ \big| \ \lvert \Es(\omega_{\q0}) \rvert = k \bigr\}
=
{\PP}^m \bigl\{ \rtheta_{\q0,k}(\omega_m) = q \ \big| \ \lvert \Es(\omega_{\q0}) \rvert = k \bigr\} ,
\]
whereas Lemma~\ref{lem:ress_properties} implies that
the event that $\rtheta_{\q0,k}(\omega_m) = q$ is independent of the event $\lvert \Es(\omega_m)\rvert = k$, so this implies
\[
{\PP}^m \bigl\{ \theta_{\q0}(\omega_m) = q \ \big| \ \lvert \Es(\omega_{\q0}) \rvert = k \bigr\}
=
{\PP}^m \bigl\{ \rtheta_{\q0,k}(\omega_m) = q \bigr\} .
\]
From the law of total probability we therefore have
\begin{align*}
{\PP}^m \bigl\{ \theta_{\q0}(\omega_m) = q \bigr\}
&=
\sum_{k=\underline{\zeta}}^{\bar{\zeta}} {\PP}^m \bigl\{ \theta_{\q0}(\omega_m) = q \ \big| \ \lvert \Es(\omega_{\q0}) \rvert = k \bigr\}
\\
&=
\sum_{k=\underline{\zeta}}^{\bar{\zeta}} {\PP}^m \bigl\{ \rtheta_{\q0,k} (\omega_m) = q \bigr\} {\PP}^m \bigl\{ \lvert \Es(\omega_m)\rvert = k\bigr\}
\end{align*}
since the events $\lvert\Es(\omega_m) \rvert = k$, $k\in\{\underline{\zeta},\ldots,\bar{\zeta}\}$ are mutually exclusive and exhaustive. 
Therefore 
\[
\min_{k\in\{\underline{\zeta},\bar{\zeta}\}} {\PP}^m \bigl\{ \rtheta_{\q0,k}(\omega_m) = q \bigr\} 
\leq
{\PP}^m \bigl\{ \theta_{\q0}(\omega_m) = q \bigr\} 
\leq 
\max_{k\in\{\underline{\zeta},\bar{\zeta}\}} {\PP}^m \bigl\{ \rtheta_{\q0,k}(\omega_m) = q \bigr\} 
\]
and the bounds (\ref{eq:prior_bound})-(\ref{eq:prior_bound_fs}) then follow 
from Lemma~\ref{lem:prior_bound}.
\qed\vv

\subsection{Confidence bounds for optimal cost values}

The confidence bounds of Theorems~\ref{thm:feas_bound} and~\ref{thm:post_feas_bound} provide information on the relationship between the optimal values of the cost functions in the chance-constrained problem~(\ref{eq:ccp}) and the sampled problem~(\ref{eq:xdef}).
This section considers these relationships and derives upper and lower bounds on the probability that the optimal cost of (\ref{eq:ccp}) is less than or equal to that of (\ref{eq:xdef}) for two cases: when $\omega$ in (\ref{eq:xdef}) is such that the solution $x^\ast(\omega)$ satisfies the constraint $f(x^\ast(\omega),\delta)\leq 0$ for a specified number of samples $\delta\in\omega_m$, and when $\omega$ is a randomly selected level-$q$ subset of $\omega_m$.
%

Let $J^o(\epsilon)$ and $J^\ast(\omega)$ denote the optimal objectives of~(\ref{eq:ccp}) and~(\ref{eq:xdef}) for given $\epsilon\in[0,1]$ and $\omega\subseteq\omega_m$, respectively.
The confidence bounds of Theorem~\ref{thm:post_feas_bound} imply the following lower bound (proof of which is provided in Appendix~\ref{apdx:proof_sec3}) on the probability that $J^o(\epsilon)$ is no greater than $J^\ast(\omega_{\q0})$ given that the solution of~(\ref{eq:xdef}) satisfies $f(x^\ast(\omega_{\q0}),\delta)\leq 0$ for exactly $q$ samples $\delta\in\omega_m$.

\begin{corollary}\label{cor:lower_costbnd}
The optimal objective values for problems~(\ref{eq:ccp}) and~(\ref{eq:xdef}) satisfy
\begin{equation}\label{eq:lower_costbnd}
\PP^m \bigl\{  J^\ast(\omega_{\q0}) \geq J^o(\epsilon) \ \big| \ \theta_{\q0}(\omega_m) = q \bigr\} 
\geq \Bincdf(q-\bar{\zeta} ; m, 1-\epsilon) 
\end{equation}
for all $\epsilon\in [0,1]$ and $\bar{\zeta}\leq \q0\leq q\leq m$.
\end{corollary}


An identical argument applied to the lower bound (\ref{eq:lower_bound_ineq}) of Theorem~\ref{thm:feas_bound} shows that the probability that $J^\ast(\omega)$ is greater than or equal to $ J^o(\epsilon)$, given that $\omega$ is a randomly selected level-$q$ subset of $\omega_m$, must likewise be at least $\Bincdf(q-\smash{\bar{\zeta}} ; m, 1-\epsilon)$.
However the upper bounds in (\ref{eq:upper_bound_ineq}) and (\ref{eq:post_upper_bound_ineq}) do not generally lead to upper bounds on the probability that $J^\ast(\omega)$ is greater than or equal to $J^o(\epsilon)$. This is because although (\ref{eq:upper_bound_ineq}) and (\ref{eq:post_upper_bound_ineq}) imply bounds on the probability that the solution of (\ref{eq:xdef})
is infeasible for (\ref{eq:ccp}), this event does not necessarily imply that $J^\ast(\omega) < J^o(\epsilon)$.
Instead we consider (following~\cite{luedtke08}) the probability that the solution of~(\ref{eq:ccp}) satisfies the constraints in~(\ref{eq:xdef}). 

\begin{theorem}\label{thm:upper_costbnd}
The optimal objective values for problems (\ref{eq:ccp}) 
and (\ref{eq:xdef}) satisfy
\begin{equation}
\PP^m\bigl\{  J^\ast(\omega_{\q0}) > J^o(\epsilon)\bigr\} \leq \Bincdf(\q0-1; r, 1-\epsilon) 
\label{eq:upper_costbnd}
\end{equation}
for all $\epsilon\in [0,1]$ and $\q0\leq m$.
\end{theorem}


Although the lower bounds on the probability of $J^\ast(\omega) \geq J^o(\epsilon)$ are identical to the lower confidence bounds of Theorems~\ref{thm:feas_bound} and~\ref{thm:post_feas_bound}, there is no such symmetry between Theorem~\ref{thm:upper_costbnd} and the upper bounds of Theorems~\ref{thm:feas_bound} and~\ref{thm:post_feas_bound}.
Furthermore, for the random discarding strategy described in Section~\ref{sec:algorithm}, the value of $q$ is likely to be larger than $\q0$ (perhaps orders of magnitude larger), and in such cases the lower bounds of Corollary~\ref{cor:lower_costbnd} change more rapidly from $0$ to $1$ as  $\epsilon$ increases
than the upper bound of Theorem~\ref{thm:upper_costbnd}. This is a consequence of the fact that $\omega_{\q0}$ is statistically identical to a randomly selected subset of $\omega_m$, rather than optimal in the sense of (\ref{eq:sp}). Therefore upper confidence bounds on the probability of $J^\ast(\omega) \geq J^o(\epsilon)$, which depend on the degree of optimality of solutions of (\ref{eq:xdef}) with respect to (\ref{eq:ccp}), are generally more conservative than the corresponding lower bounds, which depend on the probability with which solutions of
(\ref{eq:xdef}) satisfy the constraints of (\ref{eq:ccp}).

\subsection{Comparison with confidence bounds for deterministic discarding strategies}\label{sec:comparison}

The literature on robust convex programming provides various results on the probability that a sample-based approximate solution of the chance-constrained optimization problem (\ref{eq:ccp}) satisfies the chance constraints of the original problem. However, 
all of the available results relating to solution methods in which a proportion of the samples are discarded assume that deterministic discarding methods are used. 
This is because existing analyses assume 
either that the sampled problem~(\ref{eq:sp}) is solved exactly using a discarding strategy that is optimal for the given multisample, or that~(\ref{eq:sp}) is solved approximately using constraint selection strategies based on greedy or other heuristics. In this section we provide a comparison of these with the confidence bounds for randomly selected level-$q$ solutions that are given in Theorems~\ref{thm:feas_bound} and~\ref{thm:post_feas_bound} and  in Corollary~\ref{cor:lower_costbnd}, and we discuss the improvement that can be obtained using a randomised method of selecting and discarding samples.

Lower bounds on the probability that the optimal solution of (\ref{eq:sp}) satisfies the constraints of the problem in (\ref{eq:ccp}) are given in~\cite[Thm.~4.1]{calafiore10} and~\cite[Thm.~2.1]{campi11}. These bounds also apply to suboptimal constraint discarding strategies for solving (\ref{eq:sp})
based on greedy or other heuristics.
%
The analyses of~\cite{calafiore10,campi11} employ similar assumptions to those that are used here. Specifically, \cite{campi11} makes assumptions equivalent to Assumptions~\ref{ass:sp_feas} and \ref{ass:greedy_feas} (see \cite[Assump.~2.1 and 2.2]{campi11}), while Assumption \ref{ass:support_dim} is replaced by the observation that the support dimension almost surely satisfies $\bar{\zeta} \leq n$ under these conditions.
%
On the other hand \cite{calafiore10} employs Assumption~\ref{ass:greedy_feas} (see \cite[Procedure~4.1]{calafiore10}) and relaxes Assumption~\ref{ass:sp_feas} to allow for non-zero probability of infeasibility of the sampled problem~(\ref{eq:sp}). 
As a result, although \cite{calafiore10} makes an assumption that is analogous to Assumption~\ref{ass:support_dim}, this has to be modified so that the bound on the support dimension becomes $\bar{\zeta} \leq n+1$ (see \cite[Def.~3.1 and Assump.~2]{calafiore10}). With this modification, the bound in~\cite[Cor.~4.2]{calafiore10} is directly comparable to the lower confidence bounds in Theorems~\ref{thm:feas_bound} and~\ref{thm:post_feas_bound}.

Results that are equivalent to the upper bounds in (\ref{eq:upper_bound_ineq}) and (\ref{eq:post_upper_bound_ineq}) 
are given in~\cite[eq.~(7)]{campi11} and~\cite[Thm.~6.2, eq.~(6.6)]{calafiore10}. 
However the lower confidence bounds in (\ref{eq:lower_bound_ineq}), (\ref{eq:post_lower_bound_ineq}) and (\ref{eq:lower_costbnd})
are new, and the corresponding bounds in \cite{calafiore10,campi11} are not equivalent.
In particular, \cite[Cor.~4.2, eq.~(4.12)]{calafiore10}  and \cite[Thm.~2.1, eq.~(3)]{campi11} give the bound
\begin{equation}\label{eq:ccg_bnd_orig}
\PP^m \bigl\{ V\bigl(x^\ast(\omega^\ast)\bigr) > \epsilon \bigr\} \leq \binom{m - q + \bar{\zeta} - 1}{m - q} \Bincdf(m-q+\bar{\zeta}-1; m, \epsilon),
\end{equation}
where $\omega^\ast$ is the optimal level-$q$ subset of $\omega_m$ in the sampled problem (\ref{eq:sp}). 
This bound is equivalent to 
\begin{equation}\label{eq:ccg_bnd}
\PP^m\bigl\{ V\bigl(x^\ast(\omega^\ast)\bigr) \leq \epsilon \bigr\} \geq \Varcdf(q, \bar{\zeta}; m, \epsilon)
\end{equation}
where $\Varcdf$ is the function defined by
\[
\Varcdf(q, \bar{\zeta}; m, \epsilon) = 1 - \binom{m-q+\bar{\zeta} - 1}{m-q}\Bincdf(m - q + \bar{\zeta} - 1; m, \epsilon) .
\]
Similarly~\cite[Th.~6.2, eq.~(6.6)]{calafiore10} gives a bound that is equivalent to
\begin{equation}\label{eq:calaf_cost_bnd}
\PP^m\bigl\{ J^\ast(\omega^\ast) \geq J^o(\epsilon) \bigr\} \geq \Varcdf(q, \bar{\zeta}; m, \epsilon) .
\end{equation}
%

To compare (\ref{eq:lower_bound_ineq}), (\ref{eq:post_lower_bound_ineq}) and (\ref{eq:lower_costbnd}) with~(\ref{eq:ccg_bnd}) and (\ref{eq:calaf_cost_bnd}), note that, for all $q\leq m$, $\bar{\zeta} \geq 1$ and all $\epsilon\in [0,1]$, we have $\Bincdf(q-\bar{\zeta}; m, 1-\epsilon) \geq \Varcdf (q,\bar{\zeta}; m, \epsilon)$ since
\begin{align*}
1- \Varcdf (q,\bar{\zeta}; m, \epsilon) 
&= \binom{m-q+\bar{\zeta}-1}{m-q}\Bincdf(m-q+\bar{\zeta}-1;m,\epsilon)
\\
&\geq 1 - \Bincdf(q-\bar{\zeta}; m, 1-\epsilon) .
\end{align*}
Furthermore $\Bincdf(q-\bar{\zeta}; m, 1-\epsilon) = \Varcdf (q,\bar{\zeta}; m, \epsilon)$ only if $\underline{\zeta}=\bar{\zeta} =1$ or $q=m$. In both of these cases there can be at most one level-$q$ subset of $\omega_m$; if $\underline{\zeta}=\bar{\zeta} =1$ this follows from the observation that problem~(\ref{eq:xdef}) is then equivalent to a $1$-dimensional problem, whereas for $q=m$ it follows from the convexity of~(\ref{eq:xdef}).
Therefore the confidence bounds~(\ref{eq:ccg_bnd})-(\ref{eq:calaf_cost_bnd}) provided by \cite{calafiore10,campi11} are looser than the bounds (\ref{eq:lower_bound_ineq}), (\ref{eq:post_lower_bound_ineq}) and (\ref{eq:lower_costbnd}) whenever $\bar{\zeta} >1$ and $q < m$.

The discrepancy between the lower confidence bound of Theorem~\ref{thm:feas_bound} (likewise the lower bounds of Theorem~\ref{thm:post_feas_bound} and Corollary~\ref{cor:lower_costbnd})
and the corresponding bounds in~\cite{calafiore10,campi11} arises because the strategy of selecting the subset $\omega^\ast$ of $\omega_m$ that is optimal (in the sense of minimizing the objective value of~(\ref{eq:sp})) could in principle result in the worst-case probability of chance constraint satisfaction.
%
In fact any strategy employing constraint selection criteria that may be correlated with the constraint violation probability (such as, for example, suboptimal sample discarding strategies based on the objective of~(\ref{eq:sp})) are subject to the same worst case confidence bounds.

The confidence bounds of Theorems~\ref{thm:feas_bound} and~\ref{thm:post_feas_bound} are shown in Figure~\ref{fig:conf_bnds} for $m=500$, $q=375$, $\underline{\zeta} = 1$ and $\bar{\zeta} = 10$.
%
An indication of the accuracy with which $x^\ast(\omega)$, for randomly selected $\omega\in\Omega_q(\omega_m)$, approximates the solution of (\ref{eq:ccp}) 
can be obtained from the gap between the bounds $\Bincdf(q-\bar{\zeta}; m, 1-\epsilon)$ and $\Bincdf(q-\underline{\zeta}; m, 1-\epsilon)$ in Fig.~\ref{fig:conf_bnds} (the dashed black line and the solid blue line), and from the sharpness of their transition from $0$ and $1$ as $\epsilon$ increases.
Let $\underline{\epsilon}_{5}$ and $\bar{\epsilon}_{95}$ denote the values of $\epsilon$ corresponding to $5\%$ and $95\%$ confidence levels, i.e.
\[
\Bincdf(q-\underline{\zeta};m, 1-\underline{\epsilon}_{5})=0.05  
\quad \text{and} \quad 
\Bincdf(q-\bar{\zeta}; m, 1-\bar{\epsilon}_{95}) = 0.95 .
\]
Then $\bar{\epsilon}_{95} - \underline{\epsilon}_{5}$ provides a measure of the quality  of the approximate solution $x^\ast(\omega)$ 
since Boole's inequality and (\ref{eq:lower_bound_ineq})-(\ref{eq:upper_bound_ineq}) 
or (\ref{eq:post_lower_bound_ineq})-(\ref{eq:post_upper_bound_ineq}) imply
\[
\PP^m \bigl\{ \underline{\epsilon}_{5} < V(x^\ast(\omega) ) \leq \bar{\epsilon}_{95} \bigr\} \geq 0.9 .
\]
For comparison Fig.~\ref{fig:conf_bnds} shows the function $\Psi(q,\bar{\zeta};m,\epsilon)$ defining the lower bound in~(\ref{eq:ccg_bnd}), which provides a looser bound than $\Phi(q-\bar{\zeta}; m, 1-\epsilon)$.


The improvement in the quality of these confidence bounds with increasing $m$ is shown in Fig.~\ref{fig:p5_95} for the case of $\underline{\zeta} = 1$, $\bar{\zeta} = 10$, and $q = \lceil 0.75 m \rceil$.
For any value of $m\geq 200$, the range, $\bar{\epsilon}^\prime_{95} - \underline{\epsilon}_{5} $, of violation probabilities that fall within the $5\%$ and $95\%$ confidence bounds
of (\ref{eq:ccg_bnd}) and (\ref{eq:upper_bound_ineq}) is more than twice as great as the range, $\bar{\epsilon}_{95} - \underline{\epsilon}_{5}$, falling within the bounds of (\ref{eq:lower_bound_ineq}) and (\ref{eq:upper_bound_ineq}), and, as $m$ increases, the ratio $(\bar{\epsilon}^\prime_{95} - \underline{\epsilon}_{5} ) /(\bar{\epsilon}_{95} - \underline{\epsilon}_{5} )$ increases.

%

\begin{figure}
\centerline{\includegraphics[scale=0.48, trim = 28 10 42 25, clip=true]{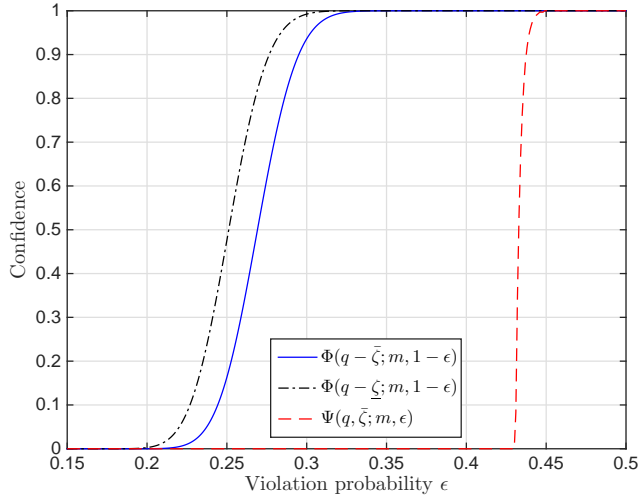}}
\caption{The confidence bounds $\Bincdf(q-\bar{\zeta};m,1-\epsilon)$ and $\Bincdf(q-\underline{\zeta}; m,1-\epsilon)$ of Theorems~\ref{thm:feas_bound} and~\ref{thm:post_feas_bound} for $m=500$ with $q=\lceil 0.75 m \rceil = 375$ and $\underline{\zeta} = 1$, $\bar{\zeta}=10$. The function $\Varcdf(q, \bar{\zeta}; m, \epsilon)$ in (\ref{eq:ccg_bnd}) is also shown for comparison.
\label{fig:conf_bnds}}
\end{figure} 

\begin{figure}
\centerline{\includegraphics[scale=0.34, trim = 0 5 35 25, clip=true]{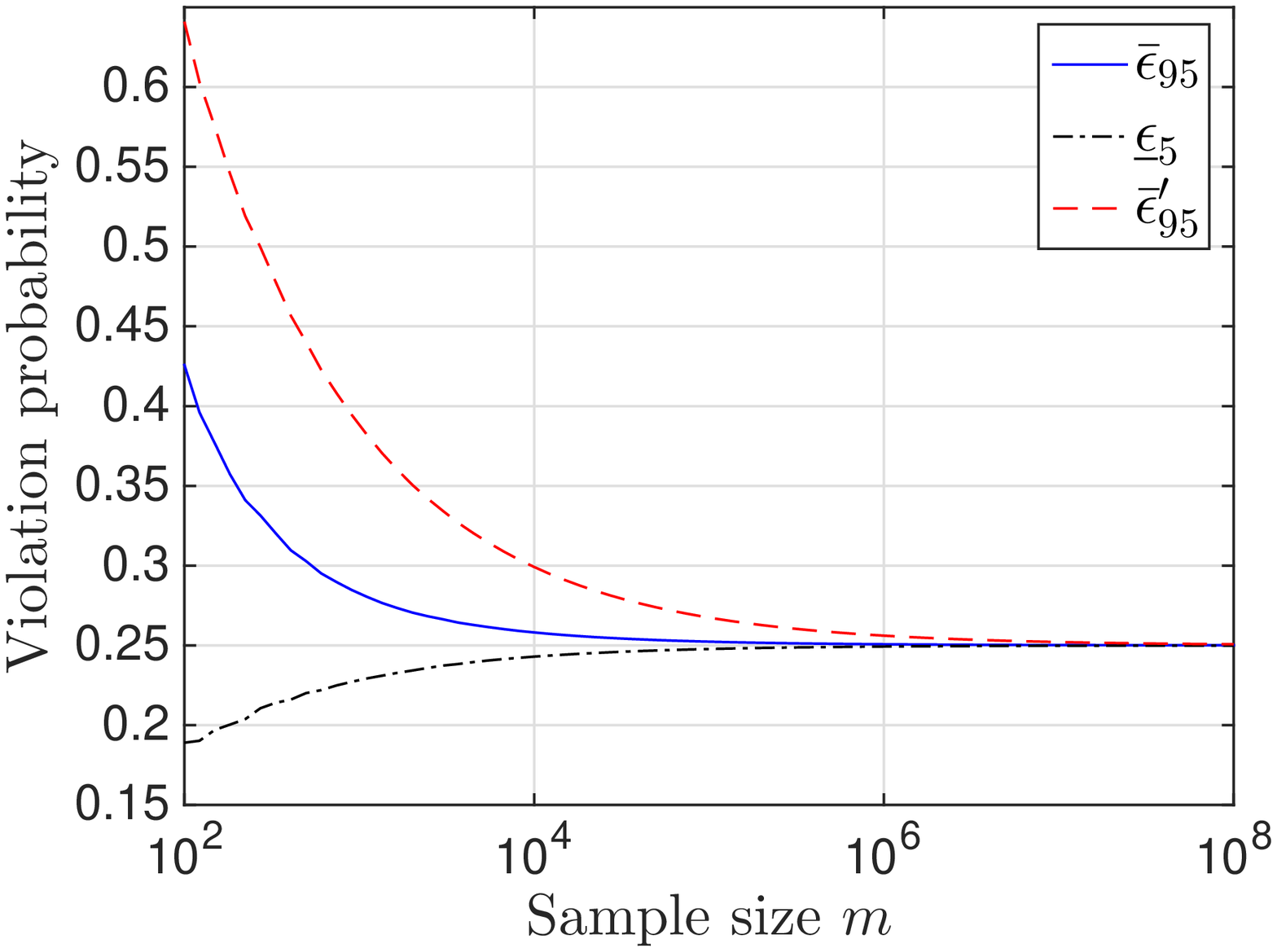}%
\includegraphics[scale=0.34, trim = 30 5 35 20, clip=true]{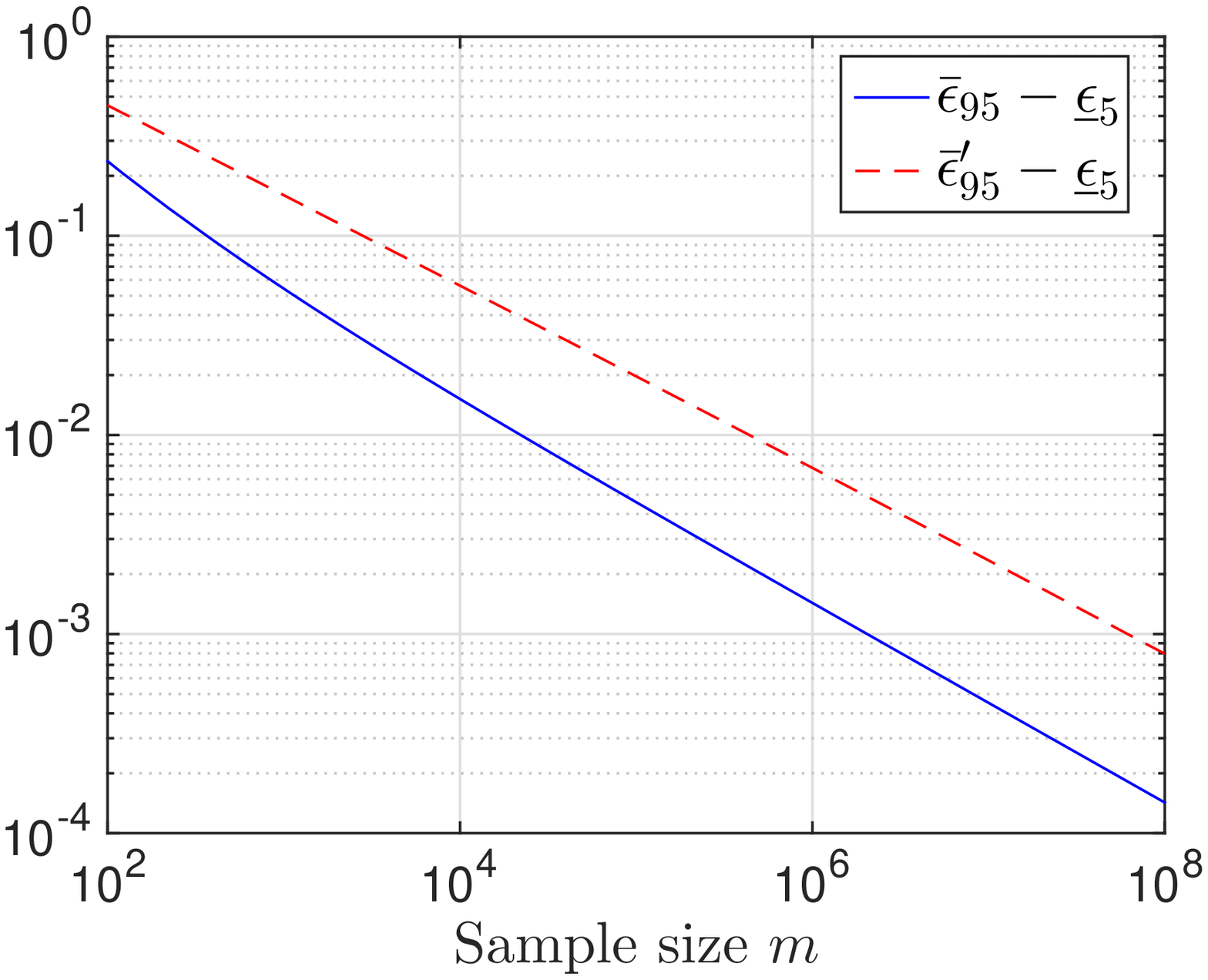}}
\caption{Values of $\epsilon$ lying on $5\%$ and $95\%$ confidence bounds 
for varying sample size $m$ with $q=\lceil 0.75 m \rceil$ and $\underline{\zeta} = 1$, $\bar{\zeta}=10$. 
The probabilities $\underline{\epsilon}_{5}$, $\bar{\epsilon}_{95}$ and $\bar{\epsilon}_{95}^\prime$ 
are defined by
$\Bincdf(q-\underline{\zeta}; m, 1-\underline{\epsilon}_{5}) = 0.05$, $\Bincdf(q-\bar{\zeta}; m, 1-\bar{\epsilon}_{95}) = 0.95$, and $\Varcdf(q,\bar{\zeta}; m,\bar{\epsilon}_{95}^\prime) = 0.95$.
\label{fig:p5_95}}
\end{figure}

\section{Random convex optimization with \textit{a priori} and \textit{a posteriori} confidence bounds}\label{sec:algorithm}


This section describes a procedure based in the results of Section~\ref{sec:results} for determining a solution of the chance-constrained problem~(\ref{eq:ccp}) with specified prior and posterior confidence bounds on violation probability.
Prior confidence bounds are imposed by requiring that the solution, $\xsol$, generated by this procedure has a violation probability $V(\xsol)$ that satisfies 
\[
  V(\xsol) \in (\underline{\epsilon},\bar{\epsilon} ]
  \text{ with probability } \pprior
\]
for given a priori bounds $\underline{\epsilon}$, $\bar{\epsilon}$ and probability $\pprior$.
The posterior bounds have the form of bounds on the probability that $V(\xsol)$ lies in an interval given knowledge of how many sampled constraints are satisfied by $\xsol$.

For any sample selection method capable of determining a randomly chosen level-$q$ subset $\omega$ of a multisample $\omega_m$, Theorem~\ref{thm:feas_bound} provides a posteriori confidence bounds on the probability that $V(\xsol)\leq \epsilon$, where $\xsol=x^\ast(\omega)$, based on the value, $\qsol$, of $q$.
%
%
%
The procedure described here ensures that $\xsol$ is statistically equivalent to a randomly selected level-$q$ solution  of~(\ref{eq:xdef}) by making use of Theorem~\ref{thm:post_feas_bound}. This results in bounds of the form (\ref{eq:post_lower_bound_ineq})-(\ref{eq:post_upper_bound_ineq}),
%
%
which are tight in the sense that they cannot be improved without information about the statistical distribution of the support dimension of~(\ref{eq:xdef}) in the interval $[\underline{\zeta},\bar{\zeta}]$. Moreover, for problems in which $\smash{\underline{\zeta} = \bar{\zeta}}$ the posterior distribution of the violation probability $V(\xsol)$ is known exactly given knowledge of $\qsol$.
Using Theorem~\ref{thm:prior_bound}, a bound can be determined on the prior probability that $\qsol$ lies in a given interval; this allows bounds on the prior distribution of $V(\xsol)$ to be computed, and provides the basis for a guarantee that $V(\xsol)$ satisfies the required a priori confidence bounds.

To exploit the confidence bounds of Theorem~\ref{thm:post_feas_bound}, the proposed method involves extracting a multisample $\omega_m$ from the distribution of $\delta$, then
solving (\ref{eq:xdef}) for $x^\ast(\omega_{r^\ast})$, for a predetermined number, $r^\ast$, of samples in $\omega_m$, before determining the number of samples $\delta\in\omega_m$ for which the constraint $f(x(\omega_{r^\ast}),\delta) \leq 0$ is satisfied.
These steps are repeated $\Ntrial$ times,
where $\Ntrial$ is chosen so as to meet the required a priori confidence bound on the solution $\xsol$. 
This procedure is similar to the approaches proposed in~\cite{chamanbaz16} and~\cite{calafiore17} for solving chance-constrained problems such as (\ref{eq:ccp}), both of which repeatedly solve a sampled convex program in the form of~(\ref{eq:xdef}), and determine, for each solution, an empirical estimate of the  associated violation probability 
in order to improve on the confidence bounds of~\cite{calafiore10,campi11}. However the procedure proposed here differs in that it uses tight bounds on the distribution of violation probability that are provided by Theorems~\ref{thm:feas_bound} and~\ref{thm:post_feas_bound} to determine prior and posterior confidence bounds.


The details of the method are given in Algorithm~\ref{alg1}. The motivation behind the design of parameters in step (i) is to ensure, with a given prior probability,
that the computation in step (ii) determines at least one level-$q$ subset of $\omega_m$ for $q\in[\underline{q},\bar{q}]$.
%
%
Therefore $r^\ast$ is chosen to maximize $\ptrial$, which, by Theorem~\ref{thm:prior_bound}, is a (tight) lower bound on the a priori probability that $\underline{q}\leq\qsol\leq\bar{q}$.
Step (ii) requires~(\ref{eq:xdef}) to be solved for $\Ntrial$ multisamples $\omega_m$, and the number of violated constraints determined in each case
(note that these $\Ntrial$ solutions and constraint violation counts may be performed in parallel).
%
The method of selecting $\xsol$ in step~(iii) ensures that $\underline{q}\leq\qsol\leq\bar{q}$ whenever the number of sampled constraints satisfied by one of the $\Ntrial$ solutions of~(\ref{eq:xdef}) computed in step~(ii) lies in the interval $[\underline{q},\bar{q}]$. 
If the solution of the minimization in step~(iii) is non-unique, the index $i^\ast$ can be randomly chosen from the minimizing set.

The value of $m$ can be chosen so that the posterior distribution of the violation probability $V(\xsol)$ has, with probability $\ppost$, a maximum uncertainty of $\Delta\epsilon$, for given $\ppost$ and $\Delta\epsilon$. This entails choosing $m$ so that the a posteriori confidence bound
\begin{equation}\label{eq:posterior_viol_bound0}
  \bigl\lvert V(\xsol) - (1-\qsol/m) \bigr\rvert \leq \Delta\epsilon
  \ \text{ with probability } \ppost
\end{equation}
holds whenever $\underline{q}\leq\qsol\leq\bar{q}$.
Using Theorem~\ref{thm:post_feas_bound}, it can be shown that a sufficient condition is that $\epsilon_a,\epsilon_b$ exist satisfying $\epsilon_a - \epsilon_b \leq \Delta \epsilon$  and 
\begin{equation}\label{eq:posterior_viol_bound_suff}
  \begin{aligned}
    \Bincdf\bigl(m(1-\bar{\epsilon}) - \bar{\zeta}; m , 1 - \epsilon_a \bigr) &\geq \tfrac{1}{2}(1+\ppost) \\
    \Bincdf\bigl(m(1-\bar{\epsilon}) - \underline{\zeta}; m , 1 - \epsilon_b \bigr) &\leq \tfrac{1}{2}(1-\ppost)  .
  \end{aligned}
\end{equation}

%


\begin{algorithm}[t]
  \caption{Procedure for solving~(\ref{eq:ccp})}
\label{alg1}
Given violation probability bounds $\underline{\epsilon}$, $\bar{\epsilon}$, probabilities $\pprior < \ppost$ and the sample size, $m$:
\begin{enumerate}[(i).]
\item
Determine $\underline{q}$, $\bar{q}$, $\qstar$ and $\ptrial$, $\Ntrial$, where
\begin{align*}
&\begin{alignedat}{3}
    \underline{q}  &= \min\ q \ &\text{ subject to } \ &\Bincdf(q -\bar{\zeta}; m , 1-\bar{\epsilon} ) \geq \tfrac{1}{2}(1+\ppost) \\
    \bar{q}  &= \max \ q  \ &\text{ subject to } \ &\Bincdf(q -\underline{\zeta}; m , 1-\underline{\epsilon} ) \leq \tfrac{1}{2}(1- \ppost) 
  \end{alignedat}\\
&  \qstar = \arg\max_{\q0} \sum_{q=\underline{q}}^{\bar{q}}
             \binom{m - \q0}{q - \q0}
             \min_{\zeta\in[\underline{\zeta},\bar{\zeta}]}
             \frac{ B(m-q+\zeta, q-\zeta + 1) }{ B(\zeta, \q0 - \zeta + 1) }\\
&  \ptrial = \sum_{q=\underline{q}}^{\bar{q}} \binom{m - \qstar}{q - \qstar}
  \min_{\zeta\in[\underline{\zeta},\bar{\zeta}]} \frac{ B(m-q+\zeta, q-\zeta + 1) }{ B(\zeta, \qstar - \zeta + 1) }
  \\
& \Ntrial = \biggl\lceil \frac{\ln (1 - \pprior/\ppost)} {\ln (1- \ptrial)} \biggr\rceil .
\end{align*}
\item
  Draw $\Ntrial$ multisamples, $\omega_m^{(i)}$, $i=1,\ldots,\Ntrial$, and for each $i$:
  \begin{enumerate}[(a)]
  \item
    solve (\ref{eq:xdef}) to determine $x^\ast(\omega_{\qstar}^{(i)})$,
  \item
    evaluate
$
\theta_{\qstar}(\omega_m^{(i)}) = \bigl\lvert \bigl\{ \delta\in\omega_m^{(i)} : f\bigl(x^\ast(\omega_{\qstar}^{(i)}), \delta \bigr) \leq 0 \bigr\} \bigr\rvert$.
\end{enumerate} 
\item
  Determine $\istar\in\{1,\ldots,\Ntrial\}$ such that $\theta_{\qstar}(\omega_m^{(\istar)})$ is closest to $\tfrac{1}{2}(\bar{q}+\underline{q})$, i.e.
  \[
    \istar = \arg\min_{i\in\{1,\ldots,\Ntrial\}} \bigl\lvert \tfrac{1}{2}(\bar{q}+\underline{q}) - \theta_{\qstar}(\omega_m^{(i)}) \bigr\rvert ,
  \]
and return the solution estimate, $\xsol=x^\ast(\omega_{\qstar}^{(\istar)})$,
 and $\qsol = \theta_{r^\ast}(\omega_m^{(\istar)})$.
\end{enumerate} 
\end{algorithm}


The following theorem provides confidence bounds on the solutions generated by Algorithm~\ref{alg1} (see Appendix~\ref{apdx:proof_sec3} for a proof).

\begin{theorem}\label{thm:alg_prior_post_bnds}
  Algorithm~\ref{alg1} generates $\xsol$ and $\qsol$ satisfying: (i) the a priori confidence bound
  \begin{equation}\label{eq:prior_viol_bound}
    \PP^{\Ntrial\, m} \bigl\{ V(\xsol) \in (\underline{\epsilon},\bar{\epsilon}\,] \bigr\} \geq \pprior ,
  \end{equation}
  (ii) the a posteriori bound~(\ref{eq:posterior_viol_bound0}) if $\qsol\in[\underline{q},\overline{q}]$; and (iii) the a posteriori bounds
  \begin{equation}\label{eq:posterior_viol_bound}
    \Bincdf(q-\bar{\zeta}; m, 1-\epsilon)
    \leq
    \PP^{\Ntrial\, m} \bigl\{ V(\xsol) \leq \epsilon \ | \ \qsol = q
    \bigr\}
    \leq
    \Bincdf(q-\underline{\zeta}; m, 1-\epsilon) 
  \end{equation}
  for all values of $q$.
\end{theorem}

The a posteriori bound (\ref{eq:posterior_viol_bound}) is tight given that the only assumption on the support dimension, $\zeta$, of (\ref{eq:xdef}) is that $\smash{\zeta\in[\underline{\zeta},\bar{\zeta}]}$. Furthermore, if $\smash{\underline{\zeta} = \bar{\zeta}}$, then (\ref{eq:posterior_viol_bound}) gives the exact posterior distribution of the violation probability conditioned on $\qsol = q$.
We note that the lower bound in~(\ref{eq:posterior_viol_bound}) coincides with the bound of~\cite[eq.~(15)]{calafiore17} for the case of fully supported problems, i.e.~if $\smash{\underline{\zeta} = \bar{\zeta}}=n$. However, Theorem~\ref{thm:alg_prior_post_bnds} shows in addition that this bound is exact in the sense that it is equal to the posterior probability of $V(\xsol)\leq \epsilon$ given that $\qsol = q$. Furthermore, Theorem~\ref{thm:alg_prior_post_bnds} provides tight posterior bounds for general $\underline{\zeta} \leq\bar{\zeta} \leq n$.

The value of $\ptrial$ is a tight bound on the probability that
$\theta_{\qstar}(\omega_m^{(i)})\in[\underline{q},\bar{q}]$, 
and it is exact if $\underline{\zeta} = \bar{\zeta}$. 
On the other hand the a priori bound (\ref{eq:prior_viol_bound}) is potentially conservative due to the use of the inequality~(\ref{eq:booles_ineq}) in the proof of Theorem~\ref{thm:alg_prior_post_bnds}. We note however that the degree of conservativeness of this a priori bound decreases as $m$ increases (so that $\Phi(q-\zeta;m,1-\epsilon)$ more closely approximates $\mathbf{1}\bigl(\epsilon - (q-\zeta)/m\bigr)$, where $\mathbf{1}(\epsilon)$ is the unit step function), and as $\bar{\epsilon}-\underline{\epsilon}$ decreases (so that $\bar{q}-\underline{q}$ decreases).

Although $\qstar$ is specified as the solution of a combinatorial optimization problem in Algorithm~\ref{alg1}, this problem can be solved effectively by exhaustive search since $\q0$ is a scalar integer variable.
However, choosing $\qstar$ so that $\ptrial$ is maximized may result in excessively large values of $\qstar$ if $\bar{\epsilon}$ is close to $0$.
In this case it may therefore be preferable to limit $\qstar$ to some maximum value at the expense of increasing the number, $\Ntrial$, of trials needed satisfy the prior probability bound (\ref{eq:prior_viol_bound}).
Table~\ref{tab:rstar_Ntrial} illustrates the variation of $\qstar$ and $\Ntrial$ with $\smash{\underline{\zeta}}$ and $\smash{\bar{\zeta}}$ for representative values of $m$, $\smash{\underline{\epsilon}}$, $\smash{\bar{\epsilon}}$, $\pprior$, $\Delta\epsilon$ and $\ppost$. Here it can be seen that, as $\smash{\underline{\zeta}}$ and $\smash{\bar{\zeta}}$ increase while the gap $\smash{\bar{\zeta}-\underline{\zeta}}$ is kept constant, the values of $\qstar$ and $\Ntrial$ respectively increase and decrease slowly. On the other hand, $\Ntrial$ grows rapidly with increasing $\bar{\zeta}$ if $\underline{\zeta}$ remains constant; this case corresponds to increasing structural uncertainty in the solution of~(\ref{eq:xdef}).


\begin{table}[t]
\centerline{\begin{tabular}{|l@{~}|@{~}c@{~}c@{~}c@{~}c@{~}|c@{~}c@{~}c@{~}c@{~}|c@{~}c@{~}c@{~}c@{~}|c@{~}c@{~}c@{~}c@{~}|c@{~}c@{~}c@{~}c@{~}|}\hline
$(\underline{\zeta},\bar{\zeta})$\rule{0pt}{11pt}&
\multicolumn{4}{c}{$(2,5)$} & 
\multicolumn{4}{|c}{$(7,10)$} & 
\multicolumn{4}{|c}{$(17,20)$} & 
\multicolumn{4}{|c}{$(47,50)$} &
\multicolumn{4}{|c|}{$(97,100)$}
\\
$\pprior$ &
$p_1$ & $p_2$ & $p_3$ & $p_4$
&
$p_1$ & $p_2$ & $p_3$ & $p_4$
&
$p_1$ & $p_2$ & $p_3$ & $p_4$
&
$p_1$ & $p_2$ & $p_3$ & $p_4$
&
$p_1$ & $p_2$ & $p_3$ & $p_4$
\\\hline
$\qstar$\rule{0pt}{11pt} &
\multicolumn{4}{c}{$15$} &
\multicolumn{4}{|c}{$40$} &
\multicolumn{4}{|c}{$91$} &
\multicolumn{4}{|c}{$241$} &
\multicolumn{4}{|c|}{$492$}
\\
$\Ntrial$ &
$84$ & $109$ & $176$ & $291$
&
$37$ & $48$ & $77$ & $128$
&
$22$ & $29$ & $46$ & $76$
&
$13$ & $16$ & $26$ & $43$
&
$8$ & $11$ & $17$ & $29$
\\ \hline
\end{tabular}}
\vspace{4mm}
\centerline{\begin{tabular}{|l@{~}|@{~}c@{~~}c@{~~}c@{~~}c@{~~}|c@{~~}c@{~~}c@{~~}c@{~~}|c@{~~}c@{~~}c@{~~}c@{~}|}\hline
$(\underline{\zeta},\bar{\zeta})$\rule{0pt}{11pt} &
\multicolumn{4}{c}{$(1,2)$} & 
\multicolumn{4}{|c}{$(1,5)$} & 
\multicolumn{4}{|c|}{$(1,10)$}
\\
$\pprior$ &
$p_1$ & $p_2$ & $p_3$ & $p_4$
&
$p_1$ & $p_2$ & $p_3$ & $p_4$
&
$p_1$ & $p_2$ & $p_3$ & $p_4$
\\\hline
$\qstar$\rule{0pt}{11pt} &
\multicolumn{4}{c}{$5$} &
\multicolumn{4}{|c}{$12$} &
\multicolumn{4}{|c|}{$22$} 
\\
$\Ntrial$ &
$96$ & $125$ & $200$ & $331$
&
$189$ & $246$ & $396$ & $655$
&
$1022$ & $1329$ & $2116$ & $3465$
\\ \hline
\end{tabular}}
\caption{Variation of $\qstar$ and $\Ntrial$ with $\underline{\zeta}$ and $\bar{\zeta}$ for $m=10^5$, $\underline{\epsilon} = 0.19$, $\bar{\epsilon} = 0.21$,
  $\Delta\epsilon = 0.005$, and $\ppost = 0.5(1+\pprior)$, and where $p_1=0.9$, $p_2=0.95$, $p_3=0.99$, $p_4 = 0.999$.\label{tab:rstar_Ntrial}
  \vspace{-2\baselineskip}
}
\end{table}

\subsection{Example: Smallest bounding hypersphere with \textit{a priori} and \textit{a posteriori} confidence bounds}\label{sec:min_sphere}


Algorithm~\ref{alg1} is applied in this section to 
the problem of determining the smallest hypersphere in $\RR^4$ that has a given probability of containing a normally distributed parameter $\delta\sim\mathcal{N}(0,I)$, with specified prior and posterior confidence bounds. Therefore we are interested in minimizing the positive scalar $R$ subject to $\PP\{ \| c - \delta \|_2 > R\} \leq \epsilon$, where $R$ and $c\in\RR^4$ are optimization variables and $\epsilon\in[0,1]$ is a given probability. 

The prior confidence bounds on the solution $\xsol$ are such that the violation probability, $V(\xsol)$ must lie between $\underline{\epsilon} = 0.19$ and $\bar{\epsilon} = 0.21$ with a probability of at least $\pprior = 0.9$. We also impose bounds on the posterior distribution by requiring that, with probability $\ppost = 0.95$, $V(\xsol)$ should differ from $1-\qsol/m$ by no more than $\pm 0.005$ whenever the solution satisfies $\qsol\in[\underline{q},\bar{q}]$.
For this problem the support dimension $\zeta$ lies almost surely between $2$ and $5$.

To check whether the posterior confidence bounds are satisfied with a multisample size of $m=10^5$ we first determine the respectively minimum and maximum values of $\epsilon_a$ and $\epsilon_b$ in (\ref{eq:posterior_viol_bound_suff}):
\[
\epsilon_a = 0.2125, \quad \epsilon_b = 0.2075.
\]
Hence (\ref{eq:posterior_viol_bound0}) holds with $\Delta \epsilon = \epsilon_a-\epsilon_b = 0.005$, which implies that the a posteriori confidence specification is met with $m=10^5$. This can be seen in Figure~\ref{fig:minball_z2_5_m1e5_post}, which shows the upper and lower posterior confidence bounds of Theorem~\ref{thm:post_feas_bound} for the extreme cases of $q = \bar{q}$ and $q = \underline{q}$.

The values of $\qstar$ and $\Ntrial$ that are needed to meet the a priori confidence specification can be determined from Table~\ref{tab:rstar_Ntrial}. Thus for $\pprior = 0.9$ and $(\underline{\zeta},\bar{\zeta}) = (2,5)$ we obtain $\qstar = 15$ and $\Ntrial=84$ (which corresponds to $\ptrial = 0.0347$). The prior confidence bound requires that the solution provided by Algorithm~\ref{alg1} should satisfy $\qsol \in [\underline{q} , \bar{q} ]$ with a probability not less than $\pprior/\ppost = 0.974$,
and, using the lower bound of Theorem~\ref{thm:prior_bound} (see Fig.~\ref{fig:minball_z2_5_m1e5_q_emp}), it can be verified that this condition is indeed met.
Figure~\ref{fig:minball_z2_5_m1e5_q_emp} shows moreover that the lower bound (\ref{eq:prior_bound}) is conservative, since the empirical distribution of $\qsol$ closely follows the upper bound given by (\ref{eq:prior_bound_fs}) in this example.
With this information the value of $\Ntrial$ can be reduced to $39$ (which corresponds to $\qstar=25$ samples and $\ptrial = 0.0736$) without violating the a priori confidence bound.

The a posteriori lower bound given in~\cite[eq.~14]{calafiore17} on the probability that $V(\xsol) \leq \epsilon$, conditioned on the value of $\qsol$, is shown in Figure~\ref{fig:minball_z2_5_m1e5_post} for the cases of $\qsol = \underline{q}$ and $\qsol = \bar{q}$ with $m=10^5$ and $\qstar = 15$. Comparing with the posterior bounds of Theorem~\ref{thm:post_feas_bound}, this lower bound can be seen (in Fig.~\ref{fig:minball_z2_5_m1e5_post}) to underestimate by $10\%$ the posterior probability of $\lvert V (\xsol) - (1-\qsol)/m \rvert \leq \Delta \epsilon$ for $\Delta \epsilon = 0.005$ (rising to $75\%$ for $\Delta\epsilon = 0.0035$).
%
Furthermore, the upper and lower bounds provided by Theorem~\ref{thm:post_feas_bound} are almost indistinguishable at this scale, implying that the posterior distribution of violation probability can be determined very precisely for this example.

\begin{figure}[ht]
  \centerline{\begin{overpic}[scale=0.48, trim = 28 10 35 0, clip=true]{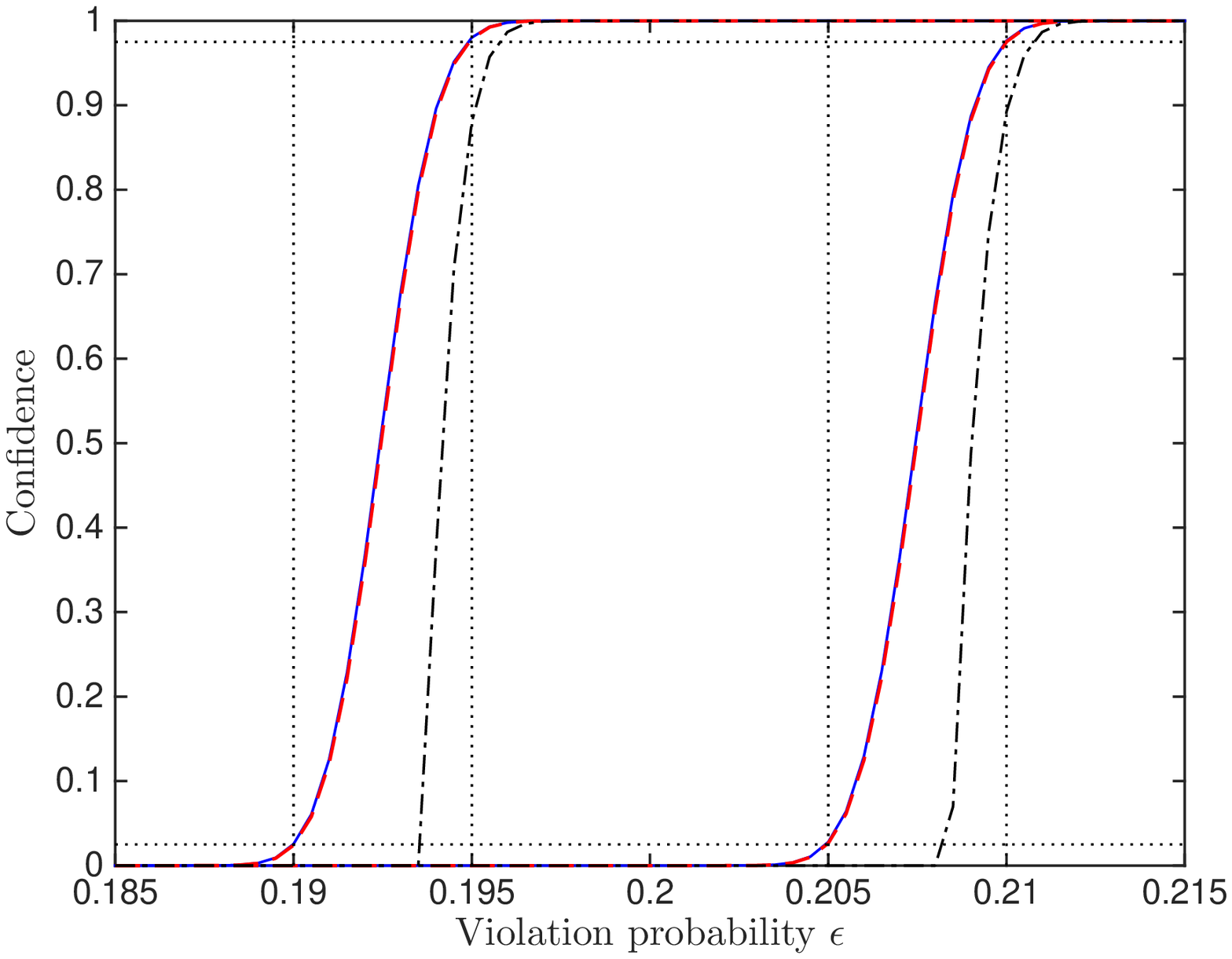}
      \put(20,78){$\epsilon = \underline{\epsilon}$}
      \put(78,78){$\epsilon = \bar{\epsilon}$}
      \put(23.5,69.15){\vector(1,0){14.7}}
      \put(30,69.15){\vector(-1,0){6.5}}
      \put(29.3,70.5){$\Delta\epsilon$}
      \put(68,69.15){\vector(1,0){13.9}}
      \put(73.7,69.15){\vector(-1,0){6.5}}
      \put(73,70.5){$\Delta\epsilon$}
      \put(28.5,18){\small{(a)}}
      \put(72.3,18){\small{(b)}}
      \put(97,9){$\tfrac{1}{2}(1-\ppost)$}
      \put(97,73.5){$\tfrac{1}{2}(1+\ppost)$}
    \end{overpic}}
  \caption{Posterior distributions of violation probability for: (a) $\hat{q} =\bar{q}$ and (b) $\hat{q} = \underline{q}$. For each case the upper and lower bounds on $\smash{\PP^{\Ntrial\,m}\bigl\{V(\xsol) \leq \epsilon \ | \ \qsol =q\bigr\}}$ in (\ref{eq:posterior_viol_bound}) are shown in blue solid lines and red dashed lines, respectively. The lower bounds of \cite[eq.~14]{calafiore17} are shown in black dash-dotted lines for $\hat{q} =\bar{q}$ and $\hat{q} = \underline{q}$.
    \label{fig:minball_z2_5_m1e5_post}}
\end{figure}

\begin{figure}[ht]
  \centerline{\begin{overpic}[scale=0.48, trim = 28 10 35 0, clip=true]{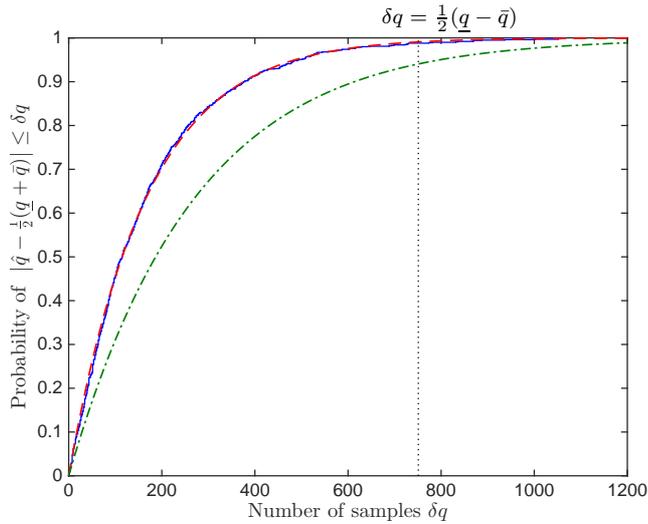}
      \put(58,78.5){$\delta q = \tfrac{1}{2}(\underline{q} -\bar{q})$}
    \end{overpic}}
  \caption{The probability that the number of unviolated constraints, $\qsol$, in Algorithm~\ref{alg1} satisfies $\bigl\lvert \hat{q} - \tfrac{1}{2}(\underline{q} + \bar{q}) \bigr\rvert \leq \delta q$. Blue solid line: empirical distribution function computed using $1000$ solutions generated by Algorithm~\ref{alg1}. The red dashed and green dash-dotted lines are respectively the upper and lower bounds given by
    Theorem~\ref{thm:prior_bound}.\label{fig:minball_z2_5_m1e5_q_emp}}
\end{figure}

\subsection{Example: Finite horizon optimal  control}\label{sec:calafex1}

The optimal control problem considered in~\cite[Sect.~IIIA]{calafiore17} is concerned with designing control sequences to drive the state of a dynamical system to a target over a finite horizon with high probability despite model uncertainty. This section considers a generalized version of the problem, which minimizes a weighted sum of probabilistic bounds on the error between the system state and the target state. The resulting controller is able to shape the probability distribution of future states more precisely than can be achieved by minimizing only a robust outer bound on the deviation from the target state. The resulting generalization therefore has potential applications to Multi-period Trading~\cite{boyd16} and Stochastic Model Predictive Control~\cite{cannon09,kouvaritakis15}.

The system has an uncertain linear discrete-time model 
\[
  z(t+1) = A(\delta) z(t)  + B u(t)
\]
with state $z\in\RR^{n_z}$, control input $u\in\RR^{n_u}$, and unknown parameter $\delta\in\Delta\subseteq\RR^{n_z\times n_z}$. The dependence of the matrix $A$ on $\delta$ is given by
\[
  A(\delta) = A_0 + \sum_{i,j=1}^{n_z} \delta_{ij}\, e_i e_j^\top
\]
where $\delta_{ij}$ is uniformly distributed on the interval $[-\rho,\rho]$ for all $i,j$,  and $e_i$ denotes the $i$th column of the $n_z\times n_z$ identity matrix.
For a horizon of $N$ steps, the control problem is to minimize the deviation of the predicted $N$-step ahead state from a target state $z_{\textit{ref}}$.
This deviation is stochastic because of model uncertainty, and a finite horizon optimal control problem is formulated to control its probability distribution.
For given $A(\delta)$ and a postulated control sequence
$u(0), \ldots, u(N-1)$,
the $N$-step ahead state, $z(u_N,\delta)$, is
\[
  z(u_N,\delta) = A^N(\delta) z(0) + \begin{bmatrix} A^{N-1}(\delta) \, B & \cdots & A(\delta)\, B & B \end{bmatrix}
  u_N
\]
where $u ^\top_N = [u^\top(0) \ \cdots \ u^\top(N-1)]$,
and we define a chance-constrained optimal control problem with decision variables
$x = \bigl(u_N,R_1,\ldots, R_\nu\bigr)\in\RR^{Nn_u+\nu}$ as:
\begin{equation}\label{eq:calafex1_ccp}
  \minimize_{x} \ \lambda\|u_N\|^2 + \sum_{j=1}^{\nu} \sigma_j R_j^2
  \ \ 
  \text{subject to}
  \ 
  \begin{aligned}[t]
    & \PP\bigl\{ \|z_{\textit{ref}} - z(u_N,\delta) \|^2 > R_j^2 \bigr\} \leq \epsilon_j  \\
    & j= 1,\ldots,\nu .
  \end{aligned}
\end{equation}
Here $\lambda$ and $\sigma_1,\ldots,\sigma_{\nu}$ are positive scalar weights and $\epsilon_1,\ldots,\epsilon_{\nu}$ are given maximum violation probabilities.  This problem formulation minimizes the weighted sum of $\nu$ probabilistic bounds on the deviation of the $N$-steps ahead state from the target state and the $l_2$-norm of the input sequence over $N$ steps.
Assuming that multisamples $\omega_{m,j} = \{ \delta_j^{(1)},\ldots,\delta_j^{(m)}\}\in\Delta^m$ for $j= 1,\ldots, \nu$ can be drawn repeatedly from the distribution for $\delta$,
we define a sampled optimization problem as:
\begin{equation}\label{eq:calafex1_sp}
  \minimize_{x} \lambda \|u_N\|^2 + \sum_{j=1}^{\nu} \sigma_j R_j^2
  \ \ 
  \text{subject to}
  \ 
  \begin{aligned}[t]
    & f_j(x,\delta) \leq 0  \text{ for all } \delta \in\omega_{r_j^\ast,j} \\
    & j = 1,\ldots,\nu ,
  \end{aligned}
\end{equation}
where $f_j$ is defined for $j=1,\ldots,\nu$ by
\[
  f_j(x,\delta) \defeq \|z_{\textit{ref}} - z(u_N,\delta) \|^2 - R_j^2 .
\]

We next describe briefly how Algorithm~\ref{alg1} can be generalized for the case of multiple chance constraints in order to determine, with a pre-specified probability $\pprior$, a solution $\hat{x}$ that satisfies the a priori condition
\begin{equation}\label{eq:multi_prior_bound}
  \PP^{\Ntrial\, m \nu} \bigl\{ V_j (\hat{x}) \in (\underline{\epsilon}_j,\bar{\epsilon}_j]
  , \ j = 1,\ldots,\nu \bigr\} \geq \pprior ,
\end{equation}
by solving (\ref{eq:calafex1_sp}) $\Ntrial$ times,
where $V_j(x)$ are violation probabilities defined for $j=1,\ldots,\nu$ by
\[
  V_j(x) \defeq \PP\bigl\{ f_j(x,\delta) > 0 \bigr\}.
\]
Specifically, by applying step (i) of Algorithm~\ref{alg1} for each $j\in\{1,\ldots,\nu\}$, the parameters $r_j^\ast$ defining each of the $\nu$ sampled constraints in~(\ref{eq:calafex1_sp}) and parameters  $\smash{\underline{q}_j}$, $\bar{q}_j$, and $p_{\textit{trial},j}$ can be computed for each $j\in\{1,\ldots,\nu\}$ given the values of $\pprior$, $p_{\textit{post},j}$, $\underline{\epsilon}_j$ and $\bar{\epsilon}_j$. Then applying step (ii) of Algorithm~\ref{alg1} we compute for each $i\in\{1,\ldots,\Ntrial\}$ the solution, $x^\ast(\omega_{r_1^\ast,1}^{(i)},\ldots, \omega_{r_\nu^\ast,\nu}^{(i)})$, of problem (\ref{eq:calafex1_sp}) and count for each $j\in\{1,\ldots,\nu\}$ the corresponding number of unviolated sampled constraints:
\[
  \theta_{r_1^\ast,\ldots,r_{\nu}^\ast,j}(\omega_{m,j}^{(i)}) \defeq
  \bigl\lvert \bigl\{ \delta \in \omega_{m,j}^{(i)} : f_j\bigl( x^\ast(\omega_{r_1^\ast,1}^{(i)},\ldots, \omega_{r_\nu^\ast,\nu}^{(i)}),\delta \bigr) \leq 0 \bigr\} \bigr\rvert .
\]
Finally, step (iii) of Algorithm~\ref{alg1} can be implemented by computing
\[
  i^\ast = \argmin_{i\in\{1,\ldots,\Ntrial\}} \max_{j\in\{1,\ldots,\nu\}} \bigl\lvert \tfrac{1}{2} (\bar{q}_j + \underline{q}_j) - \theta_{r_1^\ast,\ldots,r_{\nu}^\ast}(
  \omega_{m,j}^{(i)})\big\rvert
\]
and setting $\hat{x}\defeq x^\ast(\omega_{r_1^\ast,1}^{(i^\ast)},\ldots, \omega_{r_\nu^\ast,\nu}^{(i^\ast)})$ and $\hat{q}_j \defeq \theta_{r_1^\ast,\ldots,r_{\nu}^\ast,j}(\omega_{m,j}^{(i^\ast)})$ for $j=1,\ldots,\nu$.
Then, by using the approach of~\cite{schildbach13} to generalize Theorem~\ref{thm:alg_prior_post_bnds} for the case of multiple chance constraints, it is straightforward to show that choosing
\[
  \Ntrial = \biggl\lceil \frac{\ln (1 - \pprior/\prod_{j=1}^{\nu} p_{\textit{post},j})} {\ln (1- \prod_{j=1}^{\nu}p_{\textit{trial},j})} \biggr\rceil 
\]
ensures the prior confidence bound (\ref{eq:multi_prior_bound})
and the a posteriori bounds
\begin{equation}\label{eq:multi_posterior_bound}
  \Bincdf(q-\bar{\zeta}_j; m, 1-\epsilon)
  \leq
  \PP^{\Ntrial\, m \nu} \bigl\{ V_j(\xsol) \leq \epsilon \ | \ \qsol_j = q
  \bigr\}
  \leq
  \Bincdf(q-\underline{\zeta}_j; m, 1-\epsilon) 
\end{equation}
for $j=1,\ldots,\nu$. Here $\bar{\zeta}_j$ is the (maximum) support dimension of the constraint $f_j(x,\delta)\leq 0$ (see \cite[Def.~3.3(b)]{schildbach13}) and $\underline{\zeta}_j$ is the corresponding minimum number of sampled constraints $f_j(x,\delta)\leq 0$ that can be active at any solution of~(\ref{eq:calafex1_sp}) with non-zero probability.

We consider three versions of problem~(\ref{eq:calafex1_ccp}) with: (a) a single chance constraint with $\epsilon_1=0.005$, (b) a single chance constraint with $\epsilon_1=0.2$, and (c) two chance constraints with $\epsilon_1=0.005$ and $\epsilon_2=0.2$. 
The model has $n_z=6$ states, $n_u=1$ control variables, and $A_0$, $B$ are as given in~\cite{calafiore17}. A horizon of $N=10$ steps is employed. Note that in principle a robust bound could be imposed on $\|z_{\textit{ref}} - z(u_N,\delta)\|$ using standard robust convex programming techniques. However this would require $10^{36}$ second-order conic constraints, thus making the optimization intractable in practice.

The parameters defining the sample approximation~(\ref{eq:calafex1_sp}) in each case (a)-(c) are given in Table~\ref{tab:calafex1}.
The posterior constraint violation probabilities are required to satisfy
\[
  \bigl\lvert V_j(\xsol) - (1-\qsol_j/m) \bigr\rvert \leq \Delta\epsilon_j
  \text{ with probability } p_{\textit{post},j}
\]
whenever $\qsol_j\in[\underline{q}_j,\bar{q}_j]$, 
with tolerances of $\Delta\epsilon_1 \leq 0.005$ in (a) and (c), and $\Delta\epsilon_j \leq 0.01$ 
for $j=1$ in (b) and $j=2$ in (c). For the values of $p_{\textit{post},j}$ in Table~\ref{tab:calafex1}, a multisample size of $m=65000$ meets these requirements.
In each case the prior probability bounds (\ref{eq:multi_prior_bound}) 
are imposed with $\pprior = 0.9$.

\def\strut{\makebox[0pt][c]{\raisebox{-4pt}{\rule{0pt}{16pt}}}}
\def\strutb{\makebox[0pt][c]{\raisebox{-4pt}{\rule{0pt}{8pt}}}}
\begin{table}[ht]  \centerline{\begin{tabular}{|l@{~~}c@{~}|@{~}c@{~~}c@{~~}c@{~~}c@{~}|@{~}c@{~~}c@{~~}c@{~~}c@{~}|}\hline 
              case & $j$ & $p_{\textit{post},j}$ & $(\underline{\epsilon}_j,\bar{\epsilon}_j)$ & $(\underline{\zeta}_j,\bar{\zeta}_j)$ & $\Delta\epsilon_j$ & $\qstar_j$ & $(\underline{q}_j, \bar{q}_j)$ & $\Ntrial$ & $\ptrial$\strut\\\hline \hline
              (a) & $1$ & $1-10^{-9}$ & $(0,0.005)$ & $(1,3)$& $0.0037$ & $1000$ & $(64786,65000)$ & \ $5$ & $0.381$\strut\\\hline
              (b) & $1$ & $0.995$ & $(0.18,0.22)$ & $(1,3)$& $0.0093$ & $8$ & $(50999, 53025)$ & \ $44$ & $0.053$\strut\\\hline 
              (c) \ $\biggl\{\!\!$&\begin{tabular}{@{}c@{}} $1$\strut\\$2$\strutb\end{tabular} & \begin{tabular}{@{}c@{}} $1-10^{-9}$\strut\\$0.995$\strutb\end{tabular} & \begin{tabular}{@{}c@{}} $(0,0.005)$\strut\\$(0.18,0.22)$\strutb\end{tabular} & \begin{tabular}{@{}c@{}} $(1,3)$\strut\\$(1,3)$\strutb\end{tabular} & \begin{tabular}{@{}c@{}} $0.0037$\strut\\$0.0093$\strutb\end{tabular} & \begin{tabular}{@{}c@{}} $1000$\strut\\$8$\strutb\end{tabular} & \begin{tabular}{@{}c@{}} $(64786,65000)$\strut\\$(50999,53025)$\strutb\end{tabular} & \makebox[0pt][r]{$\biggr\}$}\ $117$ & $0.020$\strut\\ \hline
                               \end{tabular}}
          \caption{Finite horizon optimal control problem parameters for $m=65000$ and \mbox{$\pprior = 0.9$}.\label{tab:calafex1}}
  \vspace{-2\baselineskip}
\end{table}

The maximization in step (i) of Algorithm~\ref{alg1} results in a very large value of $\qstar_1$ (namely $\qstar_1 = 64786$) in (a) and (c) because $\bar{\epsilon}_1$ is close to $0$ and $p_{\textit{post},1}$ is close to~$1$. For cases (a) and (c) we therefore restrict $\qstar_1$ to a maximum of $1000$. The resulting reduction in the number of constraints in~(\ref{eq:calafex1_sp}) causes $\Ntrial$ to be only slightly greater than its minimum value (e.g.~$\Ntrial$ is increased from $1$ to $5$ in case (a)). The parameters in case (a) are similar to, and allow direct comparison with, the parameters used in~\cite[Sect.~IIIA]{calafiore17}. In particular, for any $i$, the probability that the solution of (\ref{eq:calafex1_sp}) satisfies $\theta_{r^\ast_1,1}(\omega^{(i)}) \geq \underline{q}_1$ is 
$\ptrial = 0.381$, implying that (\ref{eq:calafex1_sp}) needs to solved on average $2.6$ times before a solution meeting this condition is found. On the other hand, the bounds used in~\cite{calafiore17} give the bound on the expected number of times that (\ref{eq:calafex1_sp}) should be solved in order to obtain a solution satisfying the same proportion ($99.67\%$) of sampled constraints as $9.6$. Thus a three-fold reduction is achieved even though more samples ($2000$) are present in the counterpart to the optimization (\ref{eq:calafex1_sp}) in~\cite[Sect.~IIIA]{calafiore17}; this is a result of the improved bounds of Theorem~\ref{thm:post_feas_bound}.

The optimal control sequences for problems (a)-(c) with weights $\lambda=0.005$, $\sigma_1=1$ and $\sigma_2=10$ are shown in Figure~\ref{fig:calafex1_u}, and the associated empirical distributions of the deviations of $N$-step ahead states from the target are shown in Figures~\ref{fig:calafex1_pdf} and~\ref{fig:calafex1_cdf}. From Fig.~\ref{fig:calafex1_cdf} it can be verified that these solutions satisfy $V_j(\xsol) \in (\underline{\epsilon}_j,\bar{\epsilon}_j)$ for each $j$. The chance constraint imposed at $\epsilon_1 = 0.2$ in case (b) has the effect of reducing the mode of the $N$-step ahead state deviation relative to case (a), but this is at the expense of a heavier tail (Fig.~\ref{fig:calafex1_pdf}). The constraints at $\epsilon_1 = 0.2$ and $\epsilon_2 = 0.005$ in case (c) cause a reduction in the modal deviation with only a small increase,  relative to (a), in the $99.5\%$ bound on the deviation 
 of the $N$-step ahead state from the target.

\begin{figure}
\centerline{\includegraphics[scale=0.48, trim = 23 10 38 25, clip=true]{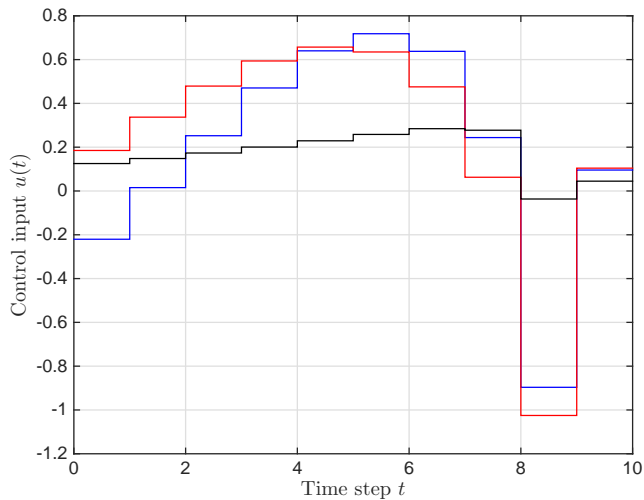}}
\caption{Optimal control sequences for the sample approximations of three instances of~(\ref{eq:calafex1_ccp}): 
  (a) single chance constraint with $\epsilon_1=0.005$ (black);
  (b) single chance constraint with $\epsilon_1=0.2$ (red);
  (c) two chance constraints with $(\epsilon_1,\epsilon_2)=(0.005,0.2)$ (blue).\label{fig:calafex1_u}}
\end{figure}

\begin{figure}
  \centerline{\begin{overpic}[scale=0.48, trim = 24 10 38 25, clip=true]{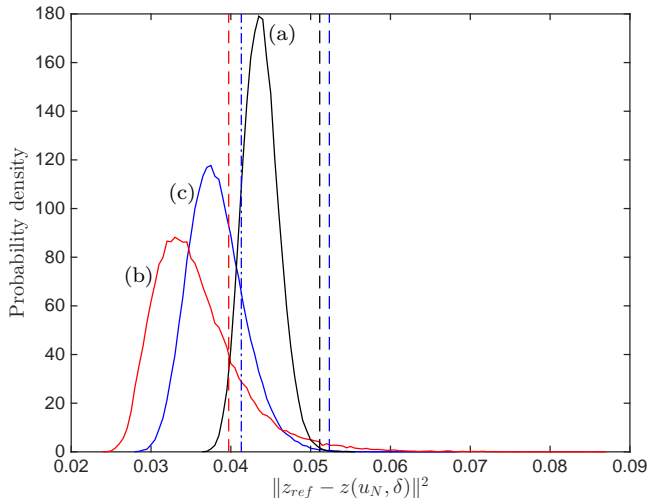}
      \put(40.5,72){(a)}
      \put(25,47){(c)}
      \put(18,34){(b)}
    \end{overpic}}
  \caption{Empirical probability densities of $\|z_{\textit{ref}} - z(u_N,\delta)\|^2$ for case (a) (black), case (b) (red), and (c) (blue).
    In each case the dashed lines show $R_1^2$ and the dash-dotted line $R_2^2$.
    \label{fig:calafex1_pdf}}
\end{figure}

\begin{figure}
  \centerline{\begin{overpic}[scale=0.48, trim = 24 10 38 25, clip=true]{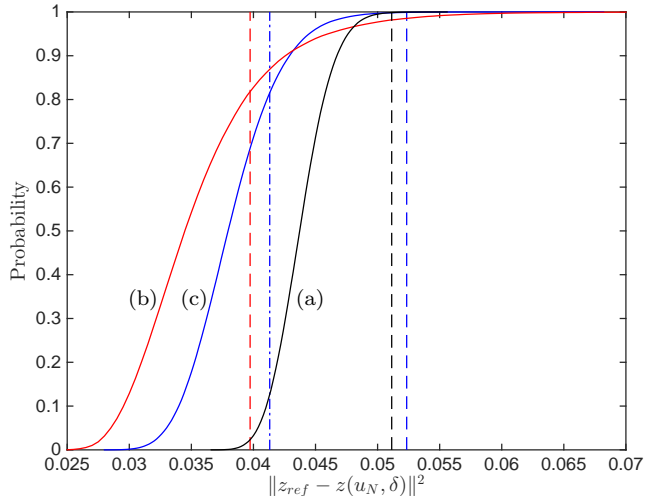}
      \put(45.5,30){(a)}
      \put(27.5,30){(c)}
      \put(19.3,30){(b)}
    \end{overpic}}
  \caption{Empirical distribution functions for $\|z_{\textit{ref}} - z(u_N,\delta)\|^2$: case (a) (black), case (b) (red), and (c) (blue). In each case the dashed lines show $R_1^2$ and the dash-dotted line $R_2^2$.
    \label{fig:calafex1_cdf}}
\end{figure}

\section{Conclusions}\label{sec:conclusions}

%
This paper  considers sample approximations of chance-con\-strained optimization problems, and shows that the use of random sample selection strategies in the definition of the sampled problem allows for tighter bounds on the confidence of feasibility with respect to the original chance constraints than sample discarding strategies based on optimality or greedy heuristics.
A randomised sample selection can be obtained by solving a sampled problem and determining the number of additionally extracted sampled constraints that are violated by the solution.
Using this observation, we propose a randomised algorithm for chance-constrained convex programming with tight a priori and a posteriori confidence bounds.
The relationships between optimal costs for the sampled problem and the chance-constrained problem are considered, and extensions of the approach to multiple chance constraints are discussed.


\appendix

\section{Omitted proofs}\label{apdx:proof_sec3}

\noindent\textbf{Proof of Lemma~\ref{lem:prior_prob}.}
For $v\in[0,1]$, let $F_{k}(v)$ denote the probability distribution of $\rV_k(\omega_{k})$ for given $k\leq m$, i.e. 
\[
F_{k}(v) = \PP^m \bigl\{ \rV_k(\omega_k) \leq v \bigr\} .
\]
Now suppose that $\rEs_k(\omega_q)$ is equal to $\omega_k$ for some $q$ such that $k\leq q\leq m$.
This event is equivalent to the event that $\delta^{(i)} \notin \rEs_k( \omega_k\cup \{\delta^{(i)}\})$ for $i=k+1,\ldots,q$,
and its probability, conditioned on the assumption that $\rV_k(\omega_k)$ is equal to $v$, is given by
\[
\PP^m \bigl\{ \omega_k = \rEs_k(\omega_q) \ \big| \ \rV_k(\omega_k) = v \bigr\} = (1-v)^{q-k}.
\] 
Using the definition of conditional probability we therefore have
\begin{equation}\label{eq:prior_prob}
\PP^m \bigl\{ \omega_k = \rEs_k(\omega_q) \cap \rV_k(\omega_k) = v  \bigr\} = (1-v)^{q-k} \, \mathrm{d} F_{k}(v) 
\end{equation}
and from the continuous version of the law of total probability it follows that 
\[
\PP^m \bigl\{ \omega_k = \rEs_k(\omega_q) \bigr\} = \int_0^1 (1-v)^{q-k} \ \mathrm{d} F_{k}(v) .
\]
But $\omega_k$ is statistically identical to a randomly selected $k$-element subset of $\omega_q$ and, from Assumption~\ref{ass:greedy_feas} and Definition~\ref{def:ress}, $\rEs_k(\omega_q)$ is almost surely unique. 
Therefore the probability that $\omega_k$ is equal to $\rEs_k(\omega_q)$ is given by the reciprocal of the number of distinct $k$-element subsets of $\omega_q$, and hence
\begin{equation}\label{eq:hausdorff_moment_prob}
\int_0^1 (1-v)^{q-k} \ \mathrm{d} F_{k}(v) = \binom{q}{k}^{-1} 
\end{equation}
necessarily holds for all $q\geq  k$.
A solution for $F_{k}$ is given by $F_{k}(v) = v^k$, and moreover it can be shown that this solution is unique (since (\ref{eq:hausdorff_moment_prob}) is equivalent to a Hausdorff moment problem~\cite[Sec.~VII.3]{feller66}).
\qed\vv



\noindent\textbf{Proof of Lemma~\ref{lem:selec_prob}}
%
  From the definition of $\rOmega_{q,k}(\omega_m)$ in (\ref{eq:reg_lq_subsets}), it follows that $\omega_k$ is equal to the regularized essential set $\rEs_k(\omega)$ for some $\omega\in\rOmega_{q,k}(\omega_m)$ if and only if $q-k$ of the $m-k$ samples $\delta$ contained in $\omega_m\setminus\omega_{k}$ satisfy $\delta\notin \rEs_k(\omega_k\cup\{\delta\})$, and the remaining $m-q$ samples satisfy $\delta\in \rEs_k(\omega_k\cup\{\delta\})$.
%
The probability of this event, conditioned on the regularized violation probability $\rV_k(\omega_{k})$ being equal to $v$, is 
\[
\PP^m \bigl\{ \omega_k = \rEs_k(\omega) \cap \omega\in\rOmega_{q,k}(\omega_m) \ \big| \ \rV_k(\omega_k) = v \bigr\} 
= \binom{m-k}{q-k} (1-v)^{q-k} v^{m-k} .
\]
Therefore, from the definition of conditional probability, we obtain
\begin{equation}\label{eq:prob_essential_and_viol_prob}
\PP^m \bigl\{ \omega_k = \rEs_k(\omega) \cap \omega\in\rOmega_{q,k}(\omega_m) \cap \rV_k(\omega_k) =v \bigr\} 
= \binom{m-k}{q-k} (1-v)^{q-k} v^{m-q} \, \mathrm{d} F_{k}(v),
\end{equation}
where $F_{k}(v) = v^k$ by Lemma~\ref{lem:prior_prob}.
Hence we have
\[
\PP^m \bigl\{ \omega_k = \rEs_k(\omega) \cap \omega\in\rOmega_{q,k}(\omega_m) 
\bigr\} =
k \binom{m-k}{q-k} \int_{0}^1 (1-v)^{q-k} v^{m-q+k-1} \, \mathrm{d} v,
\]
and the result follows from the definition of the beta function, $B(\cdot,\cdot)$ (see e.g.~\cite{abramowitz72}).
\qed\vv

\noindent\textbf{Proof of Lemma~\ref{lem:feas_bound}}
Proceeding as in the proof of Lemma~\ref{lem:selec_prob}, we first consider the probability that $\omega_k$ is the regularized essential set $\rEs_k(\omega)$ for some $\omega\in\rOmega_{q,k}(\omega_m)$ with $k\leq q \leq m$.
From~(\ref{eq:prob_essential_and_viol_prob}) and Lemma~\ref{lem:prior_prob} we obtain
\begin{align*}
\PP^m \bigl\{ \omega_k = \rEs_k(\omega) \cap \omega\in\rOmega_{q,k}(\omega_m) \cap \rV_k(\omega_k) \leq \epsilon \bigr\}
&= k\binom{m-k}{q-k} \int_0^\epsilon  (1-v)^{q-k} v^{m-q+k-1} \, \mathrm{d} v \\
&= k\binom{m-k}{q-k} B(\epsilon; m-q+k, q-k+1 ) 
\end{align*}
where $B(\cdot; \cdot, \cdot)$ is the incomplete beta function~\cite{abramowitz72}.
Using Lemma~\ref{lem:selec_prob} and the definition of conditional probability 
we obtain
\begin{align}\label{eq:cond_prob1}
\PP^m \bigl\{  \rV_k(\omega_k) \leq \epsilon 
\ | \  
\omega_k &= \rEs_k(\omega) \cap \omega \in\rOmega_{q,k}(\omega_m) \bigr\}
\\
& \quad =
\frac{
\PP^m \bigl\{ \omega_k = \rEs_k(\omega) \cap \omega\in\rOmega_{q,k}(\omega_m) \cap  \rV_k(\omega_k) \leq \epsilon \}
}{
\PP^m \bigl\{ \omega_k = \rEs_k(\omega) \cap \omega\in\rOmega_{q,k}(\omega_m) \bigr\}
}
\nonumber \\
& \quad =
\frac{B(\epsilon; m-q+k, q-k+1 )}{B(m-q+k, q-k+1 )} 
\nonumber \\
& \quad = \Bincdf(q-k; m, 1-\epsilon). 
  \nonumber
\end{align}
But the statistical independence of $(\delta^{(i)},\lambda^{(i)})$ and the index $i$ imply that the probability of $\rV_k(\omega_k)\leq \epsilon$ conditioned on $\omega_k = \rEs_k(\omega)\cap \omega \in \rOmega_{q,k}(\omega_m)$ is identical to the probability of the same event conditioned on $\omega_k = \rEs_k(\omega_q) \cap \omega_q \in\rOmega_{q,k}(\omega_m)$, so that
\begin{multline*}
  \PP^m \bigl\{  \rV_k(\omega_k) \leq \epsilon 
  \ | \  
  \omega_k = \rEs_k(\omega) \cap \omega \in\rOmega_{q,k}(\omega_m) \bigr\}
  \\ =
  \PP^m \bigl\{  \rV_k(\omega_k) \leq \epsilon 
  \ | \  
  \omega_k = \rEs_k(\omega_q) \cap \omega_q\in\rOmega_{q,k}(\omega_m) \bigr\} .
\end{multline*}
Furthermore, if $\omega_k = \rEs_k(\omega_q)$, then
$\rV_k(\omega_k) = \rV_k(\omega_q)$
almost surely, and hence
\[ 
  \PP^m \bigl\{  \rV_k(\omega_k) \leq \epsilon 
  \ | \  
  \omega_k = \rEs_k(\omega_q) \cap \omega_q\in\rOmega_{q,k}(\omega_m) \bigr\}
  =
  \PP^m \bigl\{  \rV_k(\omega_q) \leq \epsilon 
  \ | \  
  \omega_q\in\rOmega_{q,k}(\omega_m) 
  \bigr\} .
\] 
From~(\ref{eq:cond_prob1}) it therefore follows that 
$\PP^m \bigl\{ \rV_k(\omega_q) \leq \epsilon \ \big| \ \omega_q\in\rOmega_{q,k}(\omega_m) \bigr\}
=
\Bincdf(q-k; m, 1-\epsilon) .
$\qed\vv

\noindent\textbf{Proof of Lemma~\ref{lem:prior_bound}}
We begin by determining the probability distribution of $\rV_k(\omega_{\q0})$, for $k$ and $\q0$ satisfying $0\leq k \leq \q0 \leq m$. 
%
Adopting the approach used in the proof of Lemma~\ref{lem:prior_prob} and using (\ref{eq:prior_prob}) with $F_{k}(v)=v^k$ we have
\[
\PP^m \bigl\{ \omega_k = \rEs_k(\omega_{\q0}) \cap \rV_k ( \omega_k) = v \bigr\} = k (1-v)^{\q0-k} v^{k-1} \, \mathrm{d} v .
\]
But the definition of $\rOmega_{q,k}(\omega_m)$ in~(\ref{eq:reg_lq_subsets}) implies that $\omega_{\q0} = \rOmega_{\q0,k}(\omega_{\q0})$, and from Lemma~\ref{lem:selec_prob} 
therefore, the probability that $\omega_k$ is equal to the regularized essential set $\rEs_k(\omega_{\q0})$ is  $kB(k,\q0-k+1)$.
From the definition of conditional probability it follows that
\[
\PP^m \bigl\{ \rV_k (\omega_k) = v \ \big| \ \omega_k = \rEs_k(\omega_{\q0}) \bigr\}
= 
\frac{(1-v)^{\q0 - k} v^{k-1}}{B(k, \q0-k+1)}  \ \mathrm{d} v ,
\]
and since $\omega_k$ is statistically identical to a randomly selected $k$-element subset of $\omega_{\q0}$, this implies that the probability distribution of $\rV_k(\omega_{\q0})$ satisfies
\begin{equation}\label{eq:xdef_soln_viol_prob}
\PP^m \bigl\{ \rV_k (\omega_{\q0}) = v \bigr\} 
= 
\frac{(1-v)^{\q0 - k} v^{k-1} }{B(k, \q0-k+1)} \ \mathrm{d} v .
\end{equation}

This distribution can be used to determine the probability that $\rtheta_{\q0,k}(\omega_m) = q$ for $r \leq q \leq m$.
Specifically, from the definition~(\ref{eq:reg_lq_set_def}) we have that  $\smash{\rtheta_{\q0,k}}(\omega_m) = q$ if and only if
$\rEs_k(\omega_{\q0})=\rEs_k(\omega_{\q0}\cup \omega)$ for some $\omega \subseteq\omega_m\setminus \omega_{\q0}$ such that 
$\lvert \omega \rvert = q-r$,
and this occurs if and only if $q-\q0$ of the 
$m-\q0$ 
samples $\delta\in\omega_m\setminus\omega_{\q0}$ satisfy $\delta \notin \rEs_k(\omega_{\q0}\cup\{\delta\})$ and the remaining $m-q$ samples satisfy $\delta \in \rEs_k(\omega_{\q0}\cup\{\delta\})$. Hence from (\ref{eq:xdef_soln_viol_prob}), we have
\begin{equation}\label{eq:prior_bound_proof_eq}
\PP^m \bigl\{ \rtheta_{\q0,k}(\omega_m) = q \cap \rV_k (\omega_{\q0}) = v \bigr\} 
= 
\binom{m-r}{q-r} \frac{(1-v)^{q - k} v^{m-q+k-1}}{B(k, \q0-k+1)}  \, \mathrm{d} v
\end{equation}
Using the continuous version of the total probability law we therefore obtain
\[
\PP^m \bigl\{ \rtheta_{\q0,k}(\omega_m) = q \bigr\} 
= 
\binom{m-r}{q-r} \int_0^1 \frac{(1-v)^{q - k} v^{m-q+k-1}}{B(k, \q0-k+1)}  \, \mathrm{d} v 
\]
and~(\ref{eq:prior_bound_proof_eq2}) follows from the definition of the beta function, $B(\cdot,\cdot)$.
\qed\vv
\noindent\textbf{Proof of Corollary~\ref{cor:lower_costbnd}}
By optimality, $J^\ast(\omega_{\q0})$ can be no less than $J^o(\epsilon)$ if the solution $x^\ast(\omega_{\q0})$ of (\ref{eq:xdef}) satisfies the constraints of (\ref{eq:ccp}). 
Conditioned on $\theta_{\q0}(\omega_m) = q$, this implies
\[
\PP^m\bigl\{
V\bigl(x^\ast(\omega_{\q0}) \bigr) \leq \epsilon \ \big| \ \theta_{\q0}(\omega_m) = q
\bigr\}
\leq
\PP^m \bigl\{ J^\ast(\omega_{\q0}) \geq J^o(\epsilon) \ \big| \ \theta_{\q0}(\omega_m) = q \bigr\} ,
\]
which, combined with the confidence bound in~(\ref{eq:post_lower_bound_ineq}), yields~(\ref{eq:lower_costbnd}).
\qed\vv

\noindent\textbf{Proof of Theorem~\ref{thm:upper_costbnd}}
The bound (\ref{eq:upper_costbnd}) follows from the observation that the optimal objective of (\ref{eq:xdef}) must be less than or equal to that of~(\ref{eq:ccp}) whenever the solution of~(\ref{eq:ccp}) is a feasible solution for~(\ref{eq:xdef}). But the solution of~(\ref{eq:ccp}), which we denote as $x^o(\epsilon)$, satisfies the constraints of~(\ref{eq:xdef}) with $\omega=\omega_{\q0}$ if $f(x^o(\epsilon),\delta) \leq 0$ for all $\delta\in\omega_{\q0}$, and since $x^o(\epsilon)$ satisfies the constraints of~(\ref{eq:ccp}), so that $\PP\{ f(x^o(\epsilon) ,\delta) \leq 0\} \geq 1-\epsilon$, we have
\[
\PP^m\bigl\{  J^\ast(\omega_{\q0}) \leq J^o(\epsilon)\bigr\} 
\geq 
\bigl(\PP\bigl\{ f(x^o(\epsilon) ,\delta) \leq 0\bigr\}\bigr)^{\q0} 
\geq
(1-\epsilon)^{\q0}
=
1 - \Bincdf(\q0-1; \q0, 1-\epsilon) ,
\]
which implies~(\ref{eq:upper_costbnd}) because
$1-\PP^m \{ J^\ast(\omega_{\q0}) \leq J^o(\epsilon) \} = \PP^m \{ J^\ast (\omega_{\q0})>  J^o(\epsilon)]$.
\qed\vv

\noindent\textbf{Proof of Theorem~\ref{thm:alg_prior_post_bnds}}
A direct consequence of Theorem~\ref{thm:post_feas_bound} is that the solution $x^\ast(\omega_{\qstar}^{(i)})$ computed in step~(ii) necessarily satisfies, for each $i\in\{1,\ldots,\Ntrial\}$,
  \[
    \Bincdf(q-\bar{\zeta}; m , 1-\epsilon) 
    \leq
    \PP^{\Ntrial\, m} \bigl\{ V\bigl(x^\ast(\omega_{\qstar}^{(i)})\bigr) \leq \epsilon \ | \ \theta_{\qstar}(\omega_m^{(i)}) = q
    \bigr\}
    \leq
    \Bincdf(q-\underline{\zeta}; m, 1-\epsilon)  .
  \]
With $i=\istar$, these inequalities imply the posterior bounds (\ref{eq:posterior_viol_bound0})  and (\ref{eq:posterior_viol_bound}).
To demonstrate the prior bounds~(\ref{eq:prior_viol_bound}), we first evaluate the posterior bounds in~(\ref{eq:posterior_viol_bound}) for $\epsilon = \underline{\epsilon}$ and $\epsilon = \bar{\epsilon}$, and use
the definitions of $\underline{q}$ and $\bar{q}$ in step~(i) to obtain
  \begin{alignat*}{2}
    \PP^{\Ntrial\, m}\bigl\{ V(\xsol) \leq \bar{\epsilon} \ | \ \qsol = q \bigr\} &\geq \tfrac{1}{2}(1+ \ppost) &
   \ \ &\text{for all } q\geq \underline{q} \\
    \PP^{\Ntrial\, m}\bigl\{ V(\xsol) \leq \underline{\epsilon} \ | \ \qsol = q \bigr\} &\leq \tfrac{1}{2}(1- \ppost) &
   \ \ &\text{for all } q\leq \bar{q} .
 \end{alignat*}
 Hence Boole's inequality gives
 \begin{equation}\label{eq:booles_ineq}
   \PP^{\Ntrial\, m}\bigl\{ V(\xsol) \in (\underline{\epsilon}, \bar{\epsilon} \, ] \ | \ \qsol = q \bigr\} \geq \ppost \ \ \text{for all } \underline{q}\leq q \leq \bar{q} ,
 \end{equation}
and by combining this bound with the law of total probability we obtain
\begin{align*}
  \PP^{\Ntrial\, m} \bigl\{ V(\xsol) \in  (\underline{\epsilon}, \bar{\epsilon} \, ] \bigr\}
  & \geq \sum_{q = \underline{q}}^{\bar{q}}
    \PP^{\Ntrial\, m} \bigl\{ V(\xsol) \in (\underline{\epsilon}, \bar{\epsilon} \, ] \ | \ \qsol = q \bigr\} \
    \PP^{\Ntrial\, m} \bigl\{ \qsol = q \bigr\}
  \\
  & \geq \ppost\ \PP^{\Ntrial\, m} \bigl\{ \qsol \in [\underline{q} , \bar{q}] \bigr\} .
\end{align*}
But Theorem~\ref{thm:prior_bound} and the definitions of $\qstar$ and $\ptrial$ imply
$\PP^m\bigl\{ {\theta_{\qstar}(\omega_m^{(i)})} \in [\underline{q},\bar{q}]\bigr\} \geq \ptrial$ for each $i$,
%
and, noting that the definition of $\istar$ in step~(iii) implies $\qsol \in [\underline{q} , \bar{q}]$ if and only if $\smash{\theta_{\qstar}(\omega_m^{(i)})} \in  [\underline{q} , \bar{q}]$ 
 for some $i\in\{1,\ldots,\Ntrial\}$, it follows that
 \[
   \PP^{\Ntrial\, m}\bigl\{ \qsol \in [\underline{q},\bar{q}] \bigr\}
   = 1 - \PP^{\Ntrial\, m}\bigl\{ \qsol \not\in [\underline{q},\bar{q}] \bigr\}
   \geq 1 - (1-\ptrial)^{\Ntrial} 
 \]
since the multisamples $\omega^{(i)}_m$, $i=1,\ldots,\Ntrial$, are independent.
Therefore
\[
   \PP^{\Ntrial\, m} \bigl\{ V(\xsol) \in  (\underline{\epsilon}, \bar{\epsilon} \, ] \bigr\}
   \geq \ppost \bigl( 1 - (1-\ptrial)^{\Ntrial}\bigr),
 \]
 and the bound (\ref{eq:prior_viol_bound}) is implied by the definition of $\Ntrial$ in step~(ii).
\qed

\bibliographystyle{spmpsci}
\bibliography{smpc_confidencebound}

\end{document}